\documentclass[a4paper]{amsart}

\usepackage{pstricks}
\usepackage{color}
\usepackage[latin1]{inputenc}
\usepackage{amsmath,amsthm,amssymb}
\usepackage{dsfont}
\usepackage[all]{xy}
\usepackage{mathrsfs}
\usepackage{graphicx}
\usepackage{psfrag}

\newcommand{\mut}{\mu}
\newcommand{\comp}{\kappa}

\newcommand{\twist}{twist}
\newcommand{\twisting}{twisting}
\newcommand{\twists}{twists}

\newcommand{\zb}{\mathbb{Z}}

\newcommand{\al}{\alpha}
\newcommand{\eps}{\varepsilon}

\newcommand{\fl}{\rightarrow}
\newcommand{\gfl}{\longrightarrow}

\newcommand{\dr}{\ar@{->}[r]}
\newcommand{\dri}{\ar@{>->}[r]}
\newcommand{\drp}{\ar@{-->}[r]}
\newcommand{\dre}{\ar@{->>}[r]}
\newcommand{\dreg}{\ar@{=}[r]}
\newcommand{\drm}{\ar@{^{(}->}[r]}
\newcommand{\ddr}{\ar@{->}[rr]}
\newcommand{\ddre}{\ar@{->>}[rr]}
\newcommand{\ddreg}{\ar@{=}[rr]}
\newcommand{\ha}{\ar@{->}[u]}
\newcommand{\hae}{\ar@{->>}[u]}
\newcommand{\hap}{\ar@{-->}[u]}
\newcommand{\ham}{\ar@{^{(}->}[u]}
\newcommand{\hham}{\ar@{^{(}->}[uu]}
\newcommand{\hag}{\ar@{->}[ul]}
\newcommand{\hagm}{\ar@{^{(}->}[ul]}
\newcommand{\hagp}{\ar@{-->}[ul]}
\newcommand{\hdr}{\ar@{->}[ur]}
\newcommand{\hdrm}{\ar@{^{(}->}[ur]}
\newcommand{\hdri}{\ar@{>->}[ur]}
\newcommand{\hdre}{\ar@{->>}[ur]}
\newcommand{\hdrp}{\ar@{-->}[ur]}
\newcommand{\bas}{\ar@{->}[d]}
\newcommand{\bbas}{\ar@{->}[dd]}
\newcommand{\basm}{\ar@{^{(}->}[d]}
\newcommand{\basi}{\ar@{>->}[d]}
\newcommand{\base}{\ar@{->>}[d]}
\newcommand{\basp}{\ar@{-->}[d]}
\newcommand{\baseg}{\ar@{=}[d]}
\newcommand{\bbaseg}{\ar@{=}[dd]}
\newcommand{\bdr}{\ar@{->}[dr]}
\newcommand{\bdre}{\ar@{->>}[dr]}
\newcommand{\bbdr}{\ar@{->}[ddr]}
\newcommand{\bddr}{\ar@{->}[drr]}
\newcommand{\bg}{\ar@{->}[dl]}
\newcommand{\bgm}{\ar@{^{(}->}[dl]}
\newcommand{\bgp}{\ar@{-->}[dl]}
\newcommand{\bggp}{\ar@{-->}[dll]}
\newcommand{\bbgp}{\ar@{-->}[ddl]}

\newcommand{\cat}{\mathcal{C}}
\newcommand{\dc}{\mathcal{D}}

\newcommand{\rc}{\mathcal{R}}

\newcommand{\cs}{\underline{\mathcal{C}}}

\newcommand{\shift}{\Sigma}
\newcommand{\susm}{\Sigma^{-1}}

\newcommand{\ext}{\operatorname{Ext}}
\newcommand{\homph}{\operatorname{Hom}}
\newcommand{\Hom}{\operatorname{Hom}}

\newtheorem{theo}{Theorem}
\newtheorem{prop}[theo]{Proposition}
\newtheorem{lemma}[theo]{Lemma}
\newtheorem{cor}[theo]{Corollary}

\newtheorem{rk}[theo]{Remark}

\newcommand{\End}{\operatorname{End}}
\newcommand{\ens}[1]{\left\{ #1 \right\}}

\newcommand{\had}{\ar@[]}

\newcommand{\per}{\operatorname{per}}

\newcommand{\add}{\operatorname{add}}

\newcommand{\arcs}{\mathscr{R}}
\newcommand{\tri}{\mathscr{T}}

\title[Coloured quivers for rigid objects and partial triangulations]{Coloured quivers for rigid objects and partial triangulations: The unpunctured case}

\author{Robert J. Marsh}
\address{School of Mathematics, University of Leeds, Leeds, LS2 9JT, United Kingdom}
\email{marsh@maths.leeds.ac.uk}

\author{Yann Palu}
\address{Laboratoire Ami\'{e}nois
de Math\'{e}matique Fondamentale et Appliqu\'{e}e,
UFR des Sciences,
33, rue Saint-Leu,
80039 Amiens Cedex 1, France}
\email{yann.palu@u-picardie.fr}


\date{12 March 2013}


\begin{document}

\begin{abstract}
We associate a coloured quiver to a rigid object in a
Hom-finite 2-Calabi--Yau triangulated category and to a partial triangulation
on a marked (unpunctured) Riemann surface. We show that, in
the case where the category is the generalised cluster category
associated to a surface, the coloured quivers coincide.
We also show that compatible notions of mutation can be
defined and give an explicit description in the case of a
disk. A partial description is given in the general 2-Calabi--Yau case.
We show further that Iyama-Yoshino reduction can be
interpreted as cutting along an arc in the surface.
\end{abstract}

\thanks{This work was supported by the Engineering and Physical Sciences Research Council [grant number EP/G007497/1].}

\maketitle

\section*{Introduction}

Let $(S,M)$ be a pair consisting of an oriented Riemann surface $S$ with non-empty boundary and a set $M$ of marked points on the boundary of $S$, with at least one marked point on each component of the boundary. We further
assume that $(S,M)$ has no component homeomorphic to a monogon, digon, or triangle.
A \emph{partial triangulation} $\arcs$ of $(S,M)$ is a set of noncrossing simple arcs between the points in $M$. We define a mutation of such triangulations, involving replacing an arc $\al$ of $\arcs$ with a new arc depending on the surface and the rest of the partial triangulation. This allows us to associate a coloured quiver to each partial triangulation of $M$ in a natural way. The coloured quiver is a directed graph in which each edge has an associated colour which, in general, can be any integer.

Let $\cat$ be a Hom-finite, 2-Calabi--Yau, Krull-Schmidt triangulated category over a field $k$.
A \emph{rigid} object in $\cat$ is an object $R$ with no self-extensions, i.e.\
satisfying $\ext^1_{\cat}(R,R)=0$. Rigid objects in $\cat$ can also be mutated. In this case the mutation involves replacing an indecomposable summand $X$ of $R$ with a new summand depending on the relationship between $X$ and the rest of the summands of $R$. As above, this allows us to associate a coloured quiver to each rigid object of $\cat$ in a natural way.

In~\cite{BZ} the authors study the generalised cluster category $\cat_{(S,M)}$ in
the sense of Amiot~\cite{Acqw} associated to a surface
$(S,M)$ as above. Such a category is triangulated and satisfies the above
requirements. It is shown in~\cite{BZ} that, given a choice of (complete)
triangulation of $(S,M)$, there is a bijection between the simple arcs in $(S,M)$ (joining two points in $M$), up to homotopy, and the isomorphism classes of rigid indecomposable objects in $\cat_{(S,M)}$. If $X_{\al}$ denotes the object corresponding to an arc $\al$ then $\ext^1_{\cat_{(S,M)}}(X_{\al},X_{\beta})=0$
if and only if $\al$ and $\beta$ do not cross. It follows that there is a
bijection between partial triangulations of $(S,M)$ and rigid objects in
$\cat_{(S,M)}$. Our main result is that the coloured quivers defined above coincide
in this situation and that the two notions of mutation are compatible.

Suppose that $\al$ is a simple arc in $(S,M)$ as above. Let $X_{\al}$ be the indecomposable rigid object corresponding to $\al$.
Iyama-Yoshino~\cite{IY} have associated (in a more general context) a subquotient category $(\cat_{(S,M)})_{X_{\al}}$ to $X_{\al}$ which we refer to as the Iyama-Yoshino reduction of $\cat_{(S,M)}$ at $X_{\al}$. The Iyama-Yoshino reduction is again triangulated. We show that $(\cat_{(S,M)})_{X_{\al}}$ is equivalent to $\cat_{(S,M)/\al}$ where $(S,M)/\al$ denotes the new marked surface obtained from $(S,M)$ by cutting along $\al$.

By studying the combinatorics, we are able to give an explicit description of the effect of mutation on coloured quivers associated to a disk with $n$ marked points. The corresponding cluster category in this case was introduced independently in~\cite{CCS} (in geometric terms) and in~\cite{BMRRT} as the cluster category associated to
a Dynkin quiver of type $A_{n-3}$. We also give a partial explicit description of coloured quiver mutation in the general (2-Calabi--Yau) case, together with a categorical proof. In general, there are quite interesting phenomena: we give an example to show that infinitely many colours can occur in one quiver, and also show that zero-coloured $2$-cycles can occur (in contrast to the situation in~\cite{BuanThomas}).

We remark that in the case of a cluster tilting object $T$ in an acyclic cluster
category the categorical mutation we define coincides with that considered
in~\cite{BMRRT}; also with that in the $2$-Calabi-Yau case considered in~\cite{BIRS,Pgmr}.
It also coincides in the maximal rigid case considered in~\cite{GLSrigid,BIRS,IY}.
In this case, the coloured quiver we consider here encodes the same information
as the matrix associated to $T$ in~\cite{BMV} provided there are no zero-coloured two-cycles. With this restriction, the mutation of this matrix coincides~\cite[1.1]{BMV} with the mutation~\cite{FZ1}
arising in the theory of cluster algebras. We note also that this fact for the cluster tilting case was shown in~\cite{BIRS} under the assumption that there are no two-cycles or loops ($1$-cycles) in the quiver of the endomorphism algebra; the cluster category case was considered in~\cite{BMR2} and the stable module category over a preprojective algebra was considered in~\cite{GLSrigid}. See also~\cite{BIRSm} and~\cite{KY}, where
mutation of quivers with potential \cite{DWZ1} has been studied in a categorical context.
There has been a lot of work on this subject: see the survey~\cite{Kcategorification} for more details.

The geometric mutation of partial triangulations mentioned above specialises to the usual flip of an arc in the triangulation case (see~\cite[Defn.\ 3.5]{FST}).
Coloured quivers similar to those considered here have been associated to $m$-cluster tilting objects in an $(m+1)$-Calabi-Yau category in~\cite{BuanThomas} (in this case, the number of colours is fixed at $m+1$). The geometric mutation we define here should also be compared with the geometric mutation for $m$-allowable arcs in a disk~\cite[Sect.\ 11]{BuanThomas}; see also the geometric model of the $m$-cluster category of type $A$ in~\cite{BaurMarsh}.

We also note that the $2$-Calabi-Yau tilting theorem of
Keller-Reiten~\cite[Prop.\ 2.1]{KR1} (see also
Koenig-Zhu~\cite[Cor.\ 4.4]{KZ} and Iyama-Yoshino~\cite[Prop.\ 6.2]{IY})
was recently generalised~\cite{BM2} to the general rigid object case, using Gabriel-Zisman localisation. This result suggests that the mutation of general rigid
objects should be considered.

We note that some of our definitions and results could be generalised
to the punctured case, except for that fact that we rely on results
in~\cite{BZ} which apply only to the unpunctured case and are not yet known
in full generality (see the recent~\cite{CILF}).
Hence we restrict here to the unpunctured case.

The paper is organised as follows. In Section 1 we set up notation and recall the
results we need. In Section 2 we define the mutation and the coloured quiver of a rigid object in a triangulated category. In Section 3 we define mutation and the coloured quiver of a partial triangulation in a marked surface.
In Section 4 we show that cutting along an arc corresponds categorically to
Iyama-Yoshino reduction. In particular, the coloured quiver after cutting along an
arc in a partial triangulation can be obtained from the coloured quiver of the
partial triangulation by deleting a vertex.
In Section 5 we show that, for a partial triangulation of a surface and the
corresponding rigid object in the cluster category of the surface, the two notions of coloured quiver coincide. In Section $6$ we show that mutation in the type $A$
case can be described purely in terms of the coloured quiver and give an explicit
description. We also give the example mentioned above in which the associated
coloured quiver contains infinitely many colours. Finally, in Section 7, we give
a partial explicit description and categorical interpretation of coloured quiver mutation. This last result holds in any
Hom-finite, Krull-Schmidt, 2-Calabi--Yau triangulated category.

\vskip 0.3cm
\noindent \textbf{Acknowledgements}
Robert Marsh would like to thank Aslak Bakke Buan for some helpful discussions.
Both authors would like to thank the referee for helpful comments which
improved an earlier version of this article.

\section{Preliminaries}

\subsection{Riemann surfaces}
\label{ssection: Riemann surfaces}

In this section, we recall some definitions and results
from~\cite{FST} and~\cite{LF1}.

We consider a pair $(S,M)$ consisting of
an oriented Riemann surface with boundary $S$ and
a finite set $M$ of marked points on the boundary of $S$,
with at least one marked point on each boundary component.
We refer to such a pair as a \emph{marked surface}.
We fix, once and for all, an orientation of $S$, inducing the
clockwise orientation on each boundary component.

Note that:
\begin{itemize}
 \item We do not assume the surface to be connected.
 \item We only consider unpunctured marked surfaces.
\end{itemize}

We will always assume that $(S,M)$ does not have any
component homeomorphic to a monogon, a digon or a triangle.

Up to homeomorphism, each component of $(S,M)$ is determined
by the following data:
\begin{itemize}
 \item the genus $g$,
 \item the number of boundary components $b$ and
 \item the number of marked points on each boundary component $\ens{n_1,\ldots,n_b}$.
\end{itemize}

An \emph{arc} $\gamma$ in $(S,M)$ is (the isotopy class relative to endpoints of)
a curve in $S$ whose endpoints belong to $M$,
which does not intersect itself (except possibly at endpoints)
and which is not contractible to a point.
The marked points on a boundary component divide it into segments, and we say that an arc
isotopic to an arc along one of these segments is a \textit{boundary arc}.
The term \emph{arc} will usually refer to a non-boundary arc.

The set of all (non-boundary) arcs in $(S,M)$ is denoted by $A^0(S,M)$.
Two arcs are said to be \textit{non-crossing} if their isotopy classes contain representatives which do not cross, i.e.\ their crossing number is zero.
If $\arcs$ is a collection of non-crossing arcs in $(S,M)$, we will denote
by $A_\arcs^0(S,M)$ the set of arcs in $(S,M)$ which do not cross
any arc in $\arcs$ and which do not belong to $\arcs$.

A partial triangulation of $(S,M)$ is a collection of non-crossing arcs.
A maximal collection of non-crossing arcs is called a \emph{triangulation}.
The number $n$ of arcs in any triangulation of a connected marked surface is given by the formula:
$$
n = 6g + 3b + c - 6,
$$
where $c$ is the number of marked points in $M$ (see e.g.~\cite[Prop.\ 2.10]{FST}).

Let $\tri$ be a triangulation. By~\cite[Sect. 4]{FST} and~\cite[Sect. 3]{LF1}
a quiver $Q=Q_\tri$, together with a potential (a linear combination
of cycles in $Q_{\tri}$ up to cyclic permutation) $W_\tri$ can be
associated to $\tri$ as follows.
The vertices of $Q$ are the arcs of the triangulation.
There is an arrow from
$\gamma$ to $\gamma'$ for each triangle in which $\gamma'$ follows
$\gamma$ with respect to the orientation of $S$, and the potential $W_{\tri}$
is the sum of all the 3-cycles; see Figure~\ref{figure:potentialexample},
where part of a triangulated surface is shown.

\begin{figure}
\begin{center}
\psfragscanon
\psfrag{a}{$\scriptstyle\al$}
\psfrag{b}{$\scriptstyle\beta$}
\psfrag{c}{$\scriptstyle\gamma$}
\includegraphics[width=4cm]{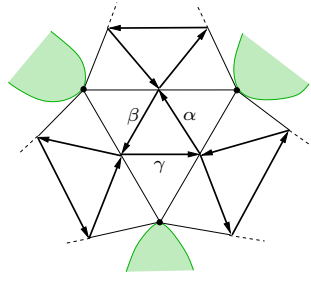}
\end{center}
\caption{The quiver with potential associated to a triangulated surface. The potential $W=\al\beta\gamma+\cdots$.}
\label{figure:potentialexample}
\end{figure}

For an arrow $a\in Q_1$, the cyclic derivative
$\partial_a$ sends a cycle $a_1\cdots a_d$
to the sum $\sum_{k=1}^d \delta_{a_k a}\, a_{k+1}\cdots a_d a_1 \cdots a_{k-1}$.
It is extended to potentials by linearity.
The {\it Jacobian algebra} of the quiver with potential $(Q_\tri,W_\tri)$
is the quotient of the complete path algebra $\widehat{kQ_{\tri}}$ by the
closure of the ideal generated by the cyclic derivatives $\partial_a W_\tri$,
for all $a\in Q_1$. We note that, by~\cite[Thm.\ 5.7]{CILF},
the Jacobian algebra can, in this case, be taken to be the quotient of the
path algebra $kQ_{\tri}$ by the ideal generated by the corresponding cyclic
derivatives in $kQ_{\tri}$.

\vspace{.2cm}
\noindent {\bf Theorem}:
{\it Let $\tri$ be a triangulation of a marked surface
$(S,M)$, and let $\tri'$ be the triangulation obtained
by flipping $\tri$ at an arc $\gamma$. Then:
\begin{itemize}
\item[(a)] \cite[Thm.\ 36]{LF1} The Jacobian algebra
$J(Q_\tri,W_\tri)$ is finite dimensional.
\item[(b)] \cite[Thm.\ 2.7]{ABCP} The Jacobian algebra
$J(Q_{\tri},W_{\tri})$ is gentle and Gorenstein of Gorenstein
dimension $1$.
 \item[(c)]\cite[Prop.\ 4.8]{FST} The quiver $Q_{\tri'}$ is
given by the Fomin--Zelevinsky mutation of $Q_\tri$
at the vertex corresponding to $\gamma$.
 \item[(d)]\cite[Thm.\ 30]{LF1} The quiver with potential
$(Q_{\tri'},W_{\tri'})$ is given by the QP mutation
(see~\cite[Sect. 5]{DWZ1}) of $(Q_\tri,W_\tri)$ at the vertex
corresponding to $\gamma$.
\end{itemize}
}

\subsection{Cluster categories associated with Riemann surfaces}

Let $K$ be a field.
If $\cat$ is a triangulated category, we will usually denote
its shift functor by $\shift$.
All the triangulated $K$-categories under consideration
in this paper are assumed to be Krull--Schmidt,
Hom-finite (all morphism spaces are finite-dimensional
 $K$-vector spaces)
and admit non-zero \textit{rigid objects} (objects $R$ such that
$\cat(R,\shift R)=0$).
All rigid objects will be assumed basic (their summands are pairwise
non-isomorphic).
We will assume moreover that the triangulated categories
are $2$-Calabi--Yau, so that there are bifunctorial isomorphisms
$\cat(X,\shift Y)\simeq D\cat(Y,\shift X)$ for all objects $X,Y$,
where $D$ is the vector space duality $D=\homph_K(-,K)$.
A rigid object $T$ is called a \textit{cluster tilting object} if, in addition,
for all objects $X$ in $\cat$, $\cat(X,\shift T) = 0 = \cat(T,\shift X)$
implies that $X$ belongs to $\add T$.

The main examples of such categories that we consider are the
(generalised) cluster categories associated with marked surfaces,
the definition of which is recalled in the following sections.

\subsubsection{Ginzburg dg algebras}

Let $(Q,W)$ be a quiver with potential (i.e.\ a QP). In this
paper, we are mostly interested in QPs arising from triangulations of
marked surfaces.

The Ginzburg dg-algebra $\Gamma(Q,W)$ is defined as follows:
First define a graded quiver $\overline{Q}$.
The vertices of $\overline{Q}$ are the vertices of $Q$,
and the arrows are given as follows:
\begin{itemize}
 \item the arrows of $Q$, of degree 0;
 \item for each arrow $\al$ in $Q$ from $i$ to $j$,
 an arrow $\al^\ast$ from $j$ to $i$, of degree $-1$;
 \item for each vertex $i$ in $Q$, a loop $e_i^\ast$ at $i$,
 of degree $-2$.
\end{itemize}
The underlying graded algebra of $\Gamma(Q,W)$
is the path algebra of the graded quiver $\overline{Q}$.
It is equipped with the unique differential $d$ sending
\begin{itemize}
 \item the arrows of degree 0 and each $e_i$ to 0;
 \item the arrow $\al^\ast$ to $\partial_\al W$, for each $\alpha\in Q_1$, and
 \item the loop $e_i^\ast$ to $e_i \left(\sum_\al [\al^\ast,\al] \right) e_i$,
 for $i\in Q_0$.
\end{itemize}

The cohomology of $\Gamma(Q,W)$ in degree zero is the Jacobian algebra $J(Q,W)$.

\subsubsection{Generalised cluster categories}
\label{ssection: generalised cluster categories}

The cluster categories associated with acyclic quivers
were introduced in~\cite{CCS} in the $A_n$ case and in~\cite{BMRRT}
in the acyclic case. Amiot defined, in~\cite{Acqw}, the
generalised cluster categories, associated with quivers
with potentials whose Jacobian algebra is finite dimensional.

Let $(Q,W)$ be a quiver with potential such that the
Jacobian algebra $J(Q,W)$ is finite dimensional,
and let $\Gamma = \Gamma(Q,W)$ be the associated Ginzburg dg algebra.

Let $\dc \Gamma$ be the derived category of $\Gamma$,
and let $\dc^b \Gamma$ be the bounded derived category.
The perfect derived category $\per \Gamma$ is
the smallest triangulated subcategory of $\dc \Gamma$
containing $\Gamma$ and stable under taking direct summands.

\vspace{.2cm}
\noindent {\bf Theorem} \cite[Sect. 6]{KDeformedCY}:
{\it
The Ginzburg dg algebra $\Gamma$ is homologically smooth
and 3-Calabi--Yau as a bimodule. In particular,
there is an inclusion $\dc^b \Gamma \subset\per \Gamma$.
}
\vspace{.2cm}

\vspace{.2cm}
\noindent {\bf Definition} \cite[Sect. 3]{Acqw}:
{\it
The (generalised) cluster category $\cat_{(Q,W)}$ associated
with the quiver with potential $(Q,W)$
is the Verdier localisation $\per \Gamma / \dc^b \Gamma$.
}
\vspace{.2cm}

This definition is motivated by the following:

\vspace{.2cm}
\noindent {\bf Theorem} \cite[Sect. 3]{Acqw}:
{\it
The cluster category $\cat_{(Q,W)}$ is Hom-finite
and 2-Calabi--Yau. Moreover, the image of $\Gamma$
in $\cat_{(Q,W)}$ is a cluster tilting object whose endomorphism
algebra is isomorphic to the Jacobian algebra $J(Q,W)$.
If $Q$ is acyclic, then $W = 0$ and the triangulated
category $\cat_{(Q,0)}$ is equivalent to the acyclic
cluster category $\cat_Q = D^b(Q)/\tau^{-1}[1]$
introduced in~\cite{BMRRT}.
}
\vspace{.2cm}

We also recall the $2$-Calabi-Yau tilting theorem which applies
in this context:

\vspace{.2cm}
\noindent {\bf Theorem} \cite[Prop.\ 2.1]{KR1}
Let $\cat$ be a triangulated Hom-finite Krull-Schmidt $2$-Calabi-Yau
category over a field $K$.
If $T$ is a cluster tilting object in $\cat$,
then the functor $\cat(T,\Sigma \,-)$ induces an equivalence between
the category $\cat/T$ and the category of finite
dimensional $\End_{\cat}(T)$-modules.
\vspace{.2cm}

Note that the assumption in the paper that $K$ be algebraically
closed is not required for this result.
We also note that this result has been generalised
(see~\cite[Prop.\ 6.2]{IY},~\cite[Sect.\ 5.1]{KZ}).

\subsubsection{Cluster categories from surfaces}
\label{ssection: cluster surfaces}

Let $(S,M)$ be a marked surface, and let
$\tri$ be a triangulation of $(S,M)$.
Let $(Q,W)$ be the quiver with potential
associated with $\tri$. The following particular
case of a theorem of Keller--Yang
shows that the cluster category $\cat_{(Q,W)}$
does not depend on the choice of a triangulation.
Let $\tri'$ be a triangulation of $(S,M)$
obtained from $\tri$ by a flip. Denote by $(Q',W')$
the associated quiver with potential.

\vspace{.2cm}
\noindent {\bf Theorem}~\cite{KY}:
{\it
There is a triangle equivalence
$\cat_{(Q',W')} \simeq \cat_{(Q,W)}$.
}
\vspace{.2cm}

Since any two triangulations of $(S,M)$ are related
by a sequence of flips, the theorem above shows
that the cluster category $\cat_{(Q,W)}$ is independent of
the choice of the triangulation $\tri$. The resulting
category is denoted $\cat_{(S,M)}$ and is called the
cluster category associated with the marked surface $(S,M)$.
(We refer also to~\cite[Theorem 5.1]{BIRSm}).

These categories have been studied by
Br\"ustle--Zhang in~\cite{BZ}.
We now recall those of their main results which will be used
in the article.

Fix a triangulation $\tri=\ens{\gamma_1,\ldots,\gamma_m}$
of $(S,M)$ with associated quiver with potential $(Q,W)$.
Let $T=T_1\oplus\cdots\oplus T_m$ be the image
of $\Gamma(Q,W)$ under $\per \Gamma(Q,W) \fl \cat_{(Q,W)} \simeq \cat_{(S,M)}$.
Note that $T$ is a cluster tilting object.

With each arc $\gamma$ not in $\tri$ is associated~\cite[Proposition 4.2]{ABCP}
an indecomposable $J(Q,W)$-module $I(\gamma)$. Let $X_{\gamma}$ be the unique
(up to isomorphism) indecomposable object in $\cat_{(S,M)}$ such that
$\cat_{(S,M)}(T,\shift X_\gamma) \simeq I(\gamma)$. Define $X_{\gamma_k} = T_k$,
for $k=1,\ldots,m$.

\vspace{.2cm}
\noindent {\bf Theorem}~\cite{BZ}:
{\it
\begin{itemize}
 \item The map $\gamma \mapsto X_\gamma$ is a bijection
 between the arcs of $(S,M)$ and the (isomorphism classes of)
 exceptional (i.e.\ indecomposable rigid) objects of $\cat_{(S,M)}$.
 \item For any two exceptional objects $X_\al$ and $X_\beta$,
 we have $\ext^1_{\cat_{(S,M)}}(X_\al,X_\beta) = 0$ if and only if
 the arcs $\al$ and $\beta$ do not cross.
 \item The shift functor of $\cat_{(S,M)}$ acts on the arcs of $(S,M)$
 by moving both endpoints \emph{clockwise} along the boundary
 to the next marked points.
\end{itemize}
}
\vspace{.2cm}

Note that a bijection with these properties is not unique in general.

We note that our choice of an orientation of the Riemann
surface differs from that of~\cite{BZ}, but coincides with
that of~\cite[Section 11]{BuanThomas}.

We extend the bijection in
the first part of the previous Theorem
to a bijection
between partial triangulations of $(S,M)$
and rigid objects in $\cat_{(S,M)}$ in the obvious way.

\subsection{Iyama--Yoshino reduction}
\label{ssection: IY reductions}

For an object
$X$ in a triangulated category $\cat$, we write ${}^\perp X$ for the full subcategory
of $\cat$ whose objects are those objects $Y$ of $\cat$ such that
$\Hom_{\cat}(Y,X)=0$. The subcategory $X^{\perp}$ is similarly defined. For an
additive subcategory $\dc$ of $\cat$, we write $\cat/[\dc]$ for the quotient category
whose objects are the same as the objects of $\cat$ with morphisms given by the
morphisms of $\cat$ modulo those morphisms factoring through $\dc$. If $\dc$ is the
additive closure of an object $X$ in $\cat$ then we just write $\cat/X$ for
$\cat/[\dc]$.

\vspace{.2cm}
\noindent {\bf Theorem}:~\cite[4.2,\,4.7]{IY}
\textit{Let $\cat$ be a $2$-Calabi-Yau triangulated category and $R$ a rigid object in $\cat$. Then the subfactor category ${}^{\perp}(\shift R)/R$ of $\cat$ is again a
$2$-Calabi-Yau triangulated category.}

We refer to the subfactor category ${}^{\perp}(\shift R)/R$ as the \emph{Iyama-Yoshino
reduction} of $\cat$ at $R$ and denote it $\cat_R$. We denote its shift by $\shift_R$ and the quotient functor ${}^{\perp}(\shift R) \gfl \cat_R$ by $\pi_R$. See also~\cite[II.2.1]{BIRS}.

We recall a result of Keller:

\vspace{.2cm}
\noindent {\bf Theorem}:~\cite[7.4]{KDeformedCY}
\textit{Let $Q,W$ be a quiver with potential
whose Jacobian algebra is finite dimensional.
Let $i$ be a vertex of $Q$ and let $\overline{P_i}$ be the image of the indecomposable projective module over $\Gamma(Q,W)$ corresponding to $i$, under the quotient functor
$\Pi:\per(\Gamma(Q,W))\to \cat_{Q,W}$. Then the Iyama-Yoshino reduction of $\cat_{(Q,W)}$
at $\overline{P_i}$ is triangle equivalent to $\cat_{(Q',W')}$, where $Q'$ is $Q$
with vertex $i$ (and all incident arrows) removed, and $W'$ is $W$ with all cycles
passing through $i$ deleted.}

\section{Coloured quivers for rigid objects}

\subsection{Mutation and coloured quivers of rigid objects}
\label{subsection: coloured quivers for rigid objects}

Let $K$ be a field.
Let $\cat$ be a $K$-linear Hom-finite, Krull--Schmidt,
$2$-Calabi--Yau triangulated $K$-category.
Let $R=R_1\oplus\cdots\oplus R_m$ be a basic rigid object in $\cat$ and
let $X$ be an indecomposable rigid object in $\cat$ Ext-orthogonal to
$R$, i.e. such that $\cat(X,\Sigma R)=0=\cat(R,\Sigma X)$.

For $c\in\zb$, consider triangles
$$
X^{(c)} \stackrel{f^c}{\gfl} B^{(c)}
\stackrel{g^c}{\gfl} X^{(c+1)} \gfl \shift X^{(c)}
$$
where $f^c$ is a minimal left $\add R$-approximation and
$g^c$ is a minimal right $\add R$-approximation
and where $X^{(0)}=X$. These will be called the
\emph{exchange triangles} for $X$ with respect to $R$.
They can be constructed using induction on $c$.
We will often write
$\comp^{(c)}_R X$ for $X^{(c)}$,
and $\comp$ for $\comp^{(1)}$; $\comp_R X$ will be referred to
as the \emph{\twist}\ of $X$ with respect to $R$.
Note that $\comp\comp^{(c)} = \comp^{(c+1)} = \comp^{(c)}\comp$
for all $c$.

These exchange triangles lift the triangles
$X^{(c)} \gfl 0 \gfl \shift_R X^{(c)} \gfl \ \shift_R X^{(c)}$
in the Iyama--Yoshino reduction $^\perp (\shift R) / R$
canonically to $\cat$.
Therefore, $X^{(c)}$ is indecomposable, rigid and $\ext$-orthogonal to $\add R$
for all $c$.
This justifies the following definition:

\vspace{.2cm}
\noindent {\bf Definition}:
{\it The \emph{mutation} of $R$ at $R_k$, for $k\in \{1,\ldots,m\}$, is the
rigid object
$$
\mut_{R_k} R = R/R_k \oplus \comp_{R/R_k} R_k.
$$
}
We will often write $\mut_k$ for $\mut_{R_k}$ and call it
the mutation at $k$.

We note that our use of the work of Iyama--Yoshino to define the mutation
above is similar to that of~\cite[Sect. 3]{BO} where cluster-tilting objects
are mutated at a non-indecomposable summand.

In~\cite{BuanThomas}, the authors associate coloured quivers to $d$-cluster--tilting
objects in ($d+1$)-Calabi--Yau categories. Here we use a similar definition
to associate a coloured quiver to $R$.

For an integer $d$, we write $\zb/d$ for the quotient of $\zb$ by the
ideal generated by $d$, identifying this with the set $\{0,1,\ldots ,d-1\}$
if $d>0$ and with $\zb$ otherwise.

\vspace{.2cm}
\noindent {\bf Definition}:
{\it The coloured quiver $Q=Q_R$ associated with the rigid object $R$
is defined as follows:
The set of vertices is $Q_0 = \ens{1,\ldots,m}$.
We label each vertex $k$ with the periodicity $d_k(R)$ (possibly infinite) of
the sequence of exchange triangles for $R_k$.
Fix two vertices $i,j$ and $c\in \mathbb{Z}/d_i(R)$. Then the number
$q_{(c)}(i,j)$ of $c$-coloured arrows from $i$ to $j$
is given by the multiplicity of $R_j$ in $B_i^{(c)}$, where
$$
R_i^{(c)} \stackrel{f_i^c}{\gfl} B_i^{(c)}
\stackrel{g_i^c}{\gfl} R_i^{(c+1)} \gfl \shift R_i^{(c)}
$$
are the exchange triangles as above for $R_i$ with respect to $R/R_i$.
}

\vspace{.2cm}
Note that by definition, $Q_R$ does not have any loops ($1$-cycles).

\vspace{.2cm}
\noindent {\bf Remark}:
\begin{itemize}
 \item Analogous definitions would apply to a functorially finite,
strictly full rigid subcategory $\rc$ of $\cat$ closed under direct sums and
direct summands, such that,
for each indecomposable $R\in\rc$, the subcategory
$\rc\setminus R$ is again functorially finite.
 \item Analogous definitions would also apply to rigid objects
in a stably $2$-Calabi--Yau Frobenius category. The use of Iyama--Yoshino
reduction would be replaced by~\cite[Theorem I.2.6]{BIRS} (see
also~\cite[Lemma 5.1]{GLSrigid} and~\cite[Sect.\ 4]{AO1}).
\end{itemize}

\subsection{Mutation of rigid objects and Iyama--Yoshino reductions}

The following lemma shows that the mutation of rigid objects
is well-behaved with respect to Iyama--Yoshino reductions.
This will turn out to be helpful in simplifying the proof
of Theorem~\ref{Theorem: mutation} in Section~\ref{section: mutation}.

Let $R=R_1\oplus \cdots \oplus R_m$ be a rigid object in $\cat$.
Let $\cat_R = \,{}^\perp(\shift R)/(R)$ be
the Iyama--Yoshino reduction of $\cat$ with respect to $R$,
with shift $\shift_R$.

Let $T$ be a rigid object in $\cat$, containing $R$ as a direct summand.
Assume that $T_k$ is a summand of $T$ but not of $R$, and
consider the exchange triangle with respect to $T/T_k$:
$$
(\ast) \hspace{.5cm}
T_k \stackrel{f}{\gfl} B_k^{(0)} \stackrel{g}{\gfl}
T_k^{(1)} \stackrel{\eps}{\gfl} \shift T_k.
$$
Here $B_k^{(0)}$ belongs to $\add \overline{T}$, where
$T=T_k\oplus \overline{T}$.

\begin{lemma}\label{lemma: reduction}
 The induced morphism $\underline{f}$ is
 a minimal left $\pi_R(\overline{T})$-approximation in $\cat_R$.
\end{lemma}

\begin{proof}
 The triangle $(\ast)$ in $\cat$ induces a triangle
 $$
T_k  \stackrel{\underline{f}}{\gfl} B_k^{(0)} \stackrel{\underline{g}}{\gfl}
T_k^{(1)} \gfl \shift_R T_k,
 $$
 in $\cat_R$. We have:
 $$
 \cat_R(\shift_R^{-1}\pi_R (T_k^{(1)}),\pi_R(\overline{T})) \simeq
 \ext^1_{\cat_R}(\pi_R(T_k^{(1)}),\pi_R(\overline{T})) \simeq
 \ext^1_\cat(T_k^{(1)},\overline{T}) = 0,
 $$
using~\cite[Lemma 4.8]{IY}.
Hence, the morphism $\underline{f}$ is a left $\pi_R(\overline{T})$-approximation.
 It is left minimal since $T_k^{(1)}$ is indecomposable in $\cat_R$.
\end{proof}

\noindent {\bf Remark}: Write $B_k^{(0)} = R_k^{(0)} \oplus C_k^{(0)}$,
with $C_k^{(0)}$ having no summands in common with $R$.
Then the morphism
$T_k \stackrel{f'}{\gfl}C_k^{(0)}$ is not a left
$\overline{T}/R$-approximation in $\cat$ in general.
\vspace{.3cm}

Let $Q$ be the coloured quiver of $T$ in $\cat$, and
let $\underline{Q}$ be the coloured quiver of $\pi_R(T)$
in $\cat_R$. Lemma~\ref{lemma: reduction} has the following
immediate corollary:
\begin{cor}\label{corollary: reduction}
 The coloured quiver $\underline{Q}$ is the full subquiver
 of $Q$ with vertices corresponding to the indecomposable summands of $T/R$.
\end{cor}

Moreover, computing the minimal $\overline{T}$-approximation
$f\in\cat$ in the triangle $(\ast)$
amounts to computing the minimal $\add\,T_j$-approximation of $T_k$
in the Iyama--Yoshino reduction $\cat_{\overline{T}/T_j}$
for all $j\neq k$. More precisely:

\begin{lemma}\label{lemma: CY approx}
Let $R=R_1\oplus\cdots\oplus R_m$ be a rigid object in $\cat$
and let $1\leq k\leq m$.
For each $j=1,\ldots,m$, let $\cat_j$ denote the Iyama--Yoshino
reduction of $\cat$ with respect to $R/(R_k\oplus R_j)$.
For $j\neq k$, let $f_j:R_k\gfl R_j^{n_j}$ be a map in $\cat$ be such that
$R_k\stackrel{\underline{f_j}}{\gfl}R_j^{n_j}$
is a minimal left $\add R_j$-approximation in $\cat_j$.
Then the morphism:
$$R_k \stackrel{\left[f_j \right]}{\gfl} \bigoplus_{j\neq k} R_j^{n_j}$$
is a minimal left $\add R/R_k$-approximation in $\cat$.
\end{lemma}

\begin{proof}
Let $i\neq k$, and let $R_k \stackrel{f}{\gfl}R_i$
be an arbitrary morphism in $\cat$.
Since $\underline{f_i}$ is an $\add R_i$-approximation in
$\cat_i$, there are morphisms
$\bigoplus_{j\neq k} R_j^{n_j} \stackrel{g_1}{\gfl} R_i$,
$R_k \stackrel{\beta_1}{\gfl} \bigoplus_{j\neq i,k} R_j^{a_j^{(1)}}$ and
$\bigoplus_{j\neq i,k} R_j^{a_j^{(1)}} \stackrel{\al_1}{\gfl} R_i$ in $\cat$
(for some $a_j^{(1)}$) such that
$f = g_1 [f_j] + \al_1 \beta_1$. Note that $\al_1$ must be a radical
map, as no summand of its domain is isomorphic to $R_i$.

Reducing to $\cat_{j}$ for some $j\not=i,k$, we see that the
component $\beta_{1,j}$ of $\beta_1$ mapping to
$R_j^{a_j^{(1)}}$ factors through $f_j:R_k \rightarrow R_j^{n_j}$
up to a map factoring through $\add \oplus_{l\not=j,k}R_{l}$.
That is, we can write $\beta_{1,j}$ as $u_jf_j+w_jv_j$
for some $u_j:R_j^{n_j}\to R_j^{a_j^{(1)}}$, $v_j:R_k\to X_j$,
and $w_j:X_j\rightarrow R_k$,
where $X_j\in \add \oplus_{l\not=j,k}R_l$. Note that
$w_j$ is a radical map, since none of the summands in $X_j$ are
isomorphic to $R_j$.

Adding over all $j$ for $j\not=i,k$, we obtain maps
$\al_2:\oplus_{l\not=k}R_l^{a_l^{(2)}}\rightarrow \oplus_{j\not=i,k}R_j^{a_j^{(1)}}$,
$\beta_2:R_k\rightarrow \oplus_{l\not=k}R_l^{a_l^{(2)}}$ and
$\gamma_2:\oplus_{j\not=k}R_j^{n_j}\rightarrow \oplus_{j\not=i,k}R_j^{a_j^{(1)}}$
(for some $a_l^{(2)}$) such that $\beta_1=\al_2\beta_2+\gamma_2[f_j]$.
Setting $g_2=\al_1\gamma_2$ we obtain
$\al_1\beta_1=\al_1\al_2\beta_2+g_2[f_j]$, so
$$f=g_1[f_j]+\al_1\beta_1=\al_1\al_2\beta_2+(g_1+g_2)[f_j].$$
Here, $\al_2$ is a radical map, since all of its summands, the $w_j$,
are radical. See Figure~\ref{figure: CY approx}.

\begin{figure}
$$
\xymatrix{
\oplus_{j\not=k} R_j^{a_j^{(r)}} \ar_{\al_r}[r] &
\cdots \ar[r] &
\oplus_{j\not=k} R_j^{a_j^{(2)}} \ar_{\al_2}[r] &
\oplus_{j\not=i,k} R_j^{a_j^{(1)}} \ar_(.6){\al_1}[r] &
R_i \\
&& \ \ \ \ \cdots \\
&&& R_k \ar_{[f_j]}[r]
\ar^{\beta_r}[llluu] \ar_{\beta_2}[luu] \ar^{\beta_1}[uu] \ar_f[uur]
& \oplus_{j\not=k} R_j^{n_j} \ar_{g_1,g_2,\ldots ,g_r}[uu]
}
$$
\caption{Proof of Lemma~\ref{lemma: CY approx}}
\label{figure: CY approx}
\end{figure}
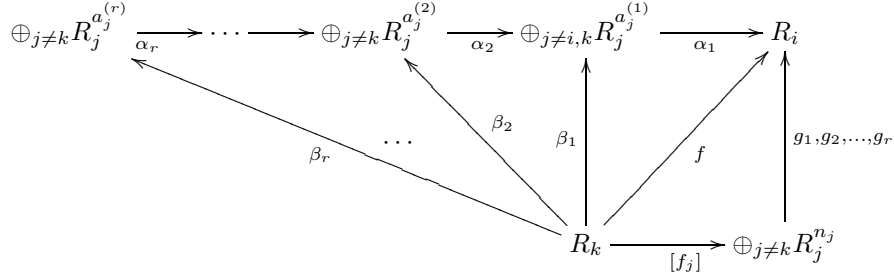

Iterating this step we construct, for all $r\geq 3$, morphisms
$g_r \,:\,\bigoplus_{j\neq k} R_j^{n_j} \gfl R_i$,
$\al_r \,:\, \bigoplus_{j\neq k} R_j^{a_j^{(r)}} \gfl
\bigoplus_{j\neq k} R_j^{a_j^{(r-1)}}$, and
$\beta_r \,:\, R_k \gfl \bigoplus_{j\neq k} R_j^{a_j^{(r)}}$
for $r=3,\ldots n$ (and some $a_l^{(r)}$), such that
$f = \beta_n \al_n \cdots \al_1 + (g_1 + \cdots + g_n)[f_j]$ and
each of the $\al_i$ is a radical map.
Since $\cat$ is Hom-finite, the radical of $\End (R)$ is nilpotent
and the composition $\beta_n \al_n \cdots \al_1$
vanishes for $n$ big enough. Therefore $f$ factors through $[f_j]$
and $[f_j]$ is a left $\add R/R_k$-approximation in $\cat$.
The left minimality of $[f_j]$ follows from the left minimality
of each $\underline{f_j}$.
\end{proof}

\section{Coloured quivers for partial triangulations}
\label{section: coloured quivers for partial triangulations}

Let $(S,M)$ be an unpunctured oriented Riemann surface with boundary and marked points. We will always assume that each boundary component contains at least one marked point and that no component of $(S,M)$ is a monogon, a digon or a triangle.

\subsection{Composition of arcs}

Let $\al$ and $\beta$ be two oriented arcs in $(S,M)$
with $\beta(1)=\al(0)$. The composition $\al\beta$
is the arc given by
$$
t \longmapsto \left\{
\begin{array}{ll}
\beta(2t) & \text{if } 0 \leq t \leq 1/2 \\
\al(2t-1) & \text{if } 1/2 \leq t \leq 1.
\end{array}
\right.
$$
See Figure~\ref{figure: composition of arcs}.

Note that the composition only makes sense for oriented arcs.

\begin{figure}
\begin{center}
\psfragscanon
\psfrag{a}{$\al$}
\psfrag{b}{$\beta$}
\psfrag{ab}{$\al\beta$}
\includegraphics[width=1.2in]{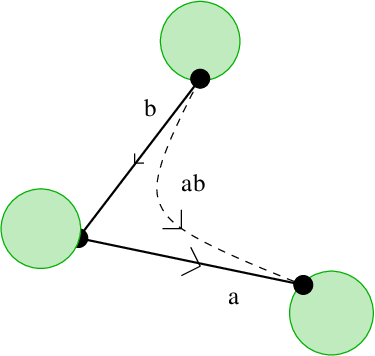}\hspace{1cm}
\includegraphics[width=3in]{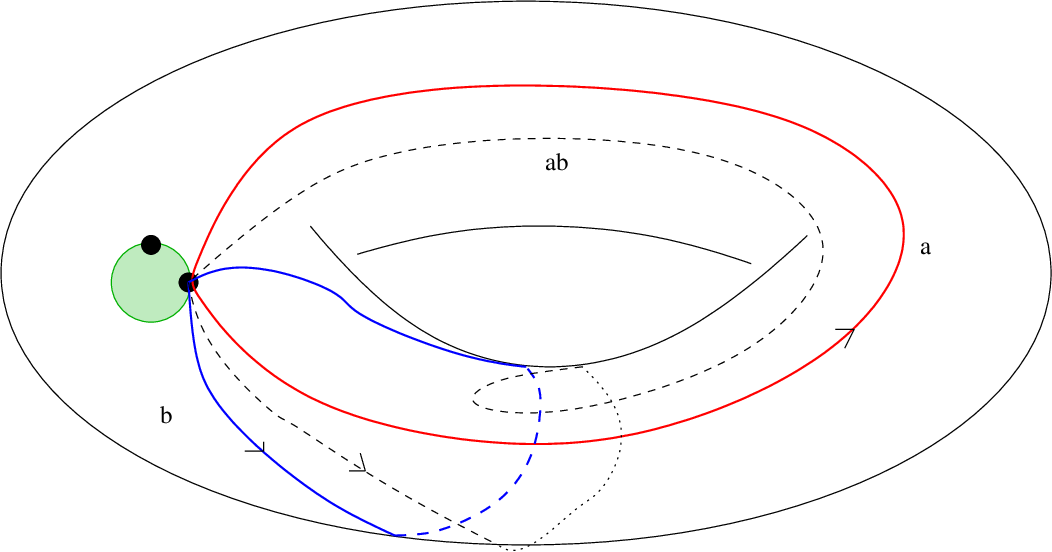}
\end{center}
\caption{Composition of arcs}\label{figure: composition of arcs}
\end{figure}

\subsection{Twisting an arc with respect to a partial triangulation}
\label{subsection: mutating arcs}

In this section, our aim is to generalise the flip of
triangulations to the \twist\ of an arc with respect to a partial triangulation.

Let $\arcs$ be a partial triangulation of $(S,M)$, i.e.\ a collection of
non-crossing arcs $\gamma_1,\ldots,\gamma_m$. Let $\al$ be an arc in $(S,M)$
which does not cross $\arcs$ and does not belong to $\arcs$, i.e.\
$\al\in A^0_{\arcs}(S,M)$.
We define the \emph{\twist}\ of $\al$ with respect to $\arcs$ as follows:
Choose an orientation $\underline{\al}$ of $\al$.
Consider the arcs of the partial triangulation $\arcs$
which admit $\underline{\al}(0)$ as an endpoint.
Restrict to a neighbourhood of $\underline{\al}(0)$ small enough
not to contain any loop.
The orientation of the boundary containing $\underline{\al}(0)$
induces an ordering on the parts of the arcs included in the neighbourhood
(see Figure~\ref{figure: twist of an arc}).
Let $\al_s$ be the arc, in $\arcs$ or boundary, which follows $\al$ in this ordering
(note that it is not allowed to be $\alpha$ itself).
Similarly, define $\al_t$ with respect to the endpoint $\al(1)$.
These will be called the \emph{arcs following $\al$ in $\arcs$}.

We give $\al_s$ and $\al_t$ the orientations $\underline{\al_s}$ and
$\underline{\al_t}$ described in the local pictures of
Figure~\ref{figure: twist of an arc}.
Note that this orientation coincides with the orientation of the boundary
if $\al_s$ or $\al_t$ is a boundary arc.

\begin{figure}
\begin{center}
\psfragscanon
\psfrag{a}{$\underline{\al}$}
\psfrag{b}{$\underline{\al_s}$}
\psfrag{c}{$\underline{\al_t}$}
\includegraphics[width=3in]{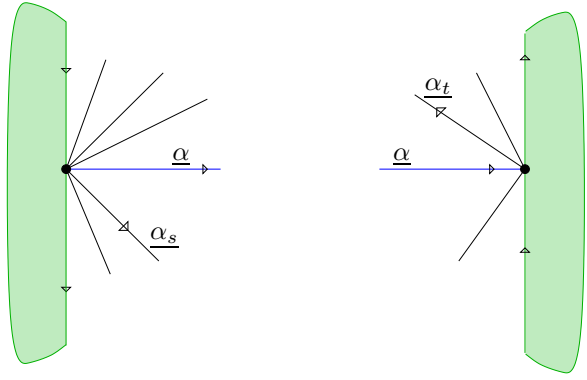}
\end{center}
\caption{Orientation of $\al_s$ and $\al_t$}\label{figure: twist of an arc}
\end{figure}

For an oriented arc $\underline{\beta}$, let $[\underline{\beta}]$ denote the
underlying unoriented arc. Define the \emph{\twist}\ of the arc $\al$ with respect to $\arcs$ to be the underlying unoriented arc of the composition:
$$
\comp_\arcs(\al) = [\underline{\al_t} \,\underline{\al}\; \underline{\al_s}^{-1}].
$$
See Figure~\ref{figure: mutation of partial triangulations}
for an example of a twist.
Note that the definition of the arc $\comp_\arcs(\al)$
does not depend on the choice of an orientation for $\al$.
It is easy to check, using a case-by-case analysis depending on whether
or not $\alpha$, $\alpha_s$, $\alpha_t$
are loops and the order in which they appear at their end-points, that
$\comp_{\arcs}(\alpha)$ does not cross any arc in $\arcs$, i.e.\ that
$\comp_{\arcs}(\alpha)\in A^0_{\arcs}(S,M)$.

\begin{figure}
\begin{center}
\psfragscanon
\psfrag{a}{$\al$}
\psfrag{b}{$\comp(\al)$}
\psfrag{c}{$\gamma_1$}
\psfrag{d}{$\gamma_2$}
\includegraphics[width=4in]{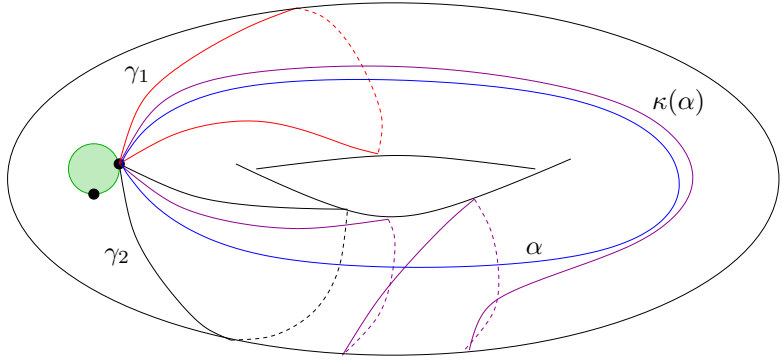}
\end{center}
\caption{An example of a \twist}\label{figure: mutation of partial triangulations}
\end{figure}

The \twist\ with respect to $\arcs$
is invertible with inverse $\comp_\arcs^{-1}$, which can
also be defined similarly.

\subsection{Mutation of partial triangulations}
\label{subsection: mutation of partial triangulations}

Let $\arcs$ be a partial triangulation of $(S,M)$,
and let $\beta\in\arcs$.
Write $\arcs = \overline{\arcs} \sqcup \ens{\beta}$.
The \emph{mutation} of $\arcs$ at $\beta$ is the partial
triangulation
$$
\mu_\beta \arcs = \overline{\arcs} \sqcup \ens{\comp_{\overline{\arcs}}(\beta)}.
$$
If $\arcs$ is a triangulation, then, for $\beta\in\arcs$,
$\mu_\beta \arcs$ is the usual flip of $\arcs$ at the arc $\beta\in\arcs$;
see~\cite[Defn.\ 3.5]{FST}.

\subsection{Coloured quivers}
\label{ssection: coloured quivers}

Let $\arcs$ be a partial triangulation and $\al$ an arc which is
not an element of $\arcs$ and does not cross $\arcs$. Let $\beta$ be
an arc in $\arcs$.
For all $c\in\zb$, define the numbers $m_\arcs^{(c)}(\al,\beta)$ by:
\begin{eqnarray*}
 m_\arcs(\al,\beta) & = &
\left\{
\begin{array}{cl}
2 & \text{if } \beta = \al_s = \al_t \\
1 & \text{if } \beta\in\ens{\al_s,\al_t} \text{ and } \al_s \neq \al_t \\
0 & \text{otherwise},
\end{array}
\right. \\
m_\arcs^{(c)}(\al,\beta) & = &
m_\arcs(\comp_\arcs^c(\al),\beta).
\end{eqnarray*}

We associate a coloured quiver $Q_\arcs$ with a partial triangulation
$\arcs=\ens{\gamma_1,\ldots,\gamma_m}$ in the following way:

\vspace{.2cm}
\noindent {\bf Definition}
{\it The coloured quiver $Q_\arcs$ associated with
the partial triangulation $\arcs$ is defined as follows:
The set of vertices is $Q_0=\ens{1,\ldots,m}$.
We label each vertex $i$ with the smallest integer $d=d_i(\arcs)$ such
that $\kappa^d_{\arcs\setminus \gamma_i}(\gamma_i)=\gamma_i$, or
with zero if no such integer exists.
Fix two distinct vertices $i,j$ and $c\in \mathbb{Z}/d_i(\arcs)$.
Then the number $q_{(c)}(i,j)$ of $c$-coloured arrows from vertex $i$ to vertex
$j$ is given by $m_{\arcs_i}^{(c)}(\gamma_i,\gamma_j)$, where
$\arcs_i = \arcs\setminus\ens{\gamma_i}$.}

Note that $Q_{\arcs}$, by definition, contains no loops.
For an example of a coloured quiver associated to a partial
triangulation of a torus and the effect of mutation on the quiver,
see Section~\ref{ssection: infinitecolour}.

\section{Cutting along an arc and CY reduction}
\label{section: cutting along an arc and CY reduction}

Let $(S,M)$ be as in Section~\ref{section: coloured quivers for partial triangulations}.

\subsection{Cutting along an arc}

Let $\al$ be an arc on $(S,M)$ not homotopic to a
point or a boundary arc. Fix a representative
of $\al$, also denoted by $\al$, whose intersection
with the boundary of $S$ consists only of its endpoints.
Then the marked surface obtained from $(S,M)$ by \textit{cutting along the arc}
$\al$ is the Riemann surface with boundary obtained by
cutting along the arc $\al$ together with the image
of the marked points $M$ after cutting.
Up to homeomorphism, it does not depend on the choice of representative
of $\al$. We will denote it by $(S,M)/\al$.
Note that if $\al$ is not a loop, then each endpoint of $\al$ gives rise to two distinct marked points in $(S,M)/\al$. If $\al$ is a loop, its endpoint
gives rise to three distinct marked points in $(S,M)/\al$.

The resulting marked surface cannot contain a monogon as a
connected component, since $\al$ is not homotopic to a point.
No connected component can be a bigon, since $\al$ is not a boundary arc.
If a component homeomorphic to a triangle has been created, we remove it.

There is a natural bijection between the arcs on
$(S,M)/\al$ and the arcs of $(S,M)$ which do not cross
the arc $\al$. Moreover, the (partial) triangulations of
$(S,M)/\al$ correspond, through this bijection, to the
(partial) triangulations of $(S,M)$ containing the arc $\al$.

\begin{rk}\label{remark: gluing}
The surface $(S,M)/\al$ can also be constructed as follows.
Let $\tri$ be a triangulation of $(S,M)$ containing $\al$.
The surface $(S,M)$ is then obtained from the triangles
of the triangulation by gluing matching sides of triangles
in a prescribed orientation. The surface $(S,M)/\al$
is obtained from the same triangles by respecting the same gluings
except for the sides which correspond to $\al$, which are not glued together anymore.
\end{rk}

Given a collection $\arcs$ of non-crossing arcs,
one can cut successively along each arc. Whatever order
is chosen yields the same new surface, by Remark~\ref{remark: gluing}.
The corresponding surface will be called the \emph{reduction}
of $(S,M)$ with respect to $\arcs$, and will be denoted
by $(S,M)/\arcs$. We will denote the natural bijection between
$A^0_{\arcs}(S,M)$ and $A^0((S,M)/\arcs)$ by $\pi_{\arcs}$.

\subsection{Compatibility with CY reduction}

Let $R$ be a basic rigid object in $\cat_{(S,M)}$,
and let $\arcs$ be the associated partial triangulation.
We denote by $\cat_R =\,\!\!^\perp(\shift R)/(R)$
the Calabi--Yau reduction of $\cat_{(S,M)}$
with respect to $R$, and by $(S,M)/\arcs$
the marked surface obtained from $(S,M)$
by cutting along the arcs of $\arcs$.

\begin{prop}\label{proposition: CY}
The triangulated categories
$\cat_{(S,M)/\arcs}$ and $\cat_R$ are equivalent.
\end{prop}

\begin{proof}
Complete the collection of arcs $\arcs$ to a triangulation
$\tri$. Let $(Q,W)$ be the QP associated with $\tri$.
By definition, there is an equivalence of triangulated categories
$\cat_{(S,M)} \simeq \cat_{(Q,W)}$.
By~\cite[Theorem 7.4]{KDeformedCY} (see section~\ref{ssection: IY reductions}),
the category $\cat_R$ is triangle equivalent to
the cluster category $\cat_{(Q',W')}$, where
$(Q',W')$ is obtained from $(Q,W)$ by deleting
the vertices of $Q$ which correspond to arcs in $\arcs$,
and all adjacent arrows.
On the other hand, the arcs in $\tri$ not in $\arcs$ induce a triangulation
of the surface $(S,M)/\arcs$. It follows from
Remark~\ref{remark: gluing} that $(Q',W')$
is the QP associated with this triangulation.
Thus $\cat_{(S,M)/\arcs}$ is equivalent to $\cat_{(Q',W')}$.
\end{proof}

\vspace{.2cm}
\noindent {\bf Remark}:
Lemma~\ref{lemma: reductioncorrespondence} shows that the equivalence above
is well-behaved with respect to well-chosen bijections between arcs and
exceptional objects.
\vspace{.2cm}

Figure~\ref{figure:qpexample1} shows the effect of cutting along an arc
in a triangulation of a torus with a single boundary component containing two marked
points. We cut along the red arc (numbered $3$) and obtain a cylinder with four
marked points as shown, with triangulation given by the remaining arcs.
In the last step, the cylinder has been rotated around to get a simpler picture.
The effect on the corresponding quiver with potential is shown in Figure~\ref{figure:qpexample2}.

\begin{figure}
\begin{center}
\psfragscanon
\psfrag{1}{$\scriptstyle 1$}
\psfrag{2}{$\scriptstyle 2$}
\psfrag{3}{$\scriptstyle 3$}
\psfrag{4}{$\scriptstyle 4$}
\psfrag{5}{$\scriptstyle 5$}
\includegraphics[scale=0.5]{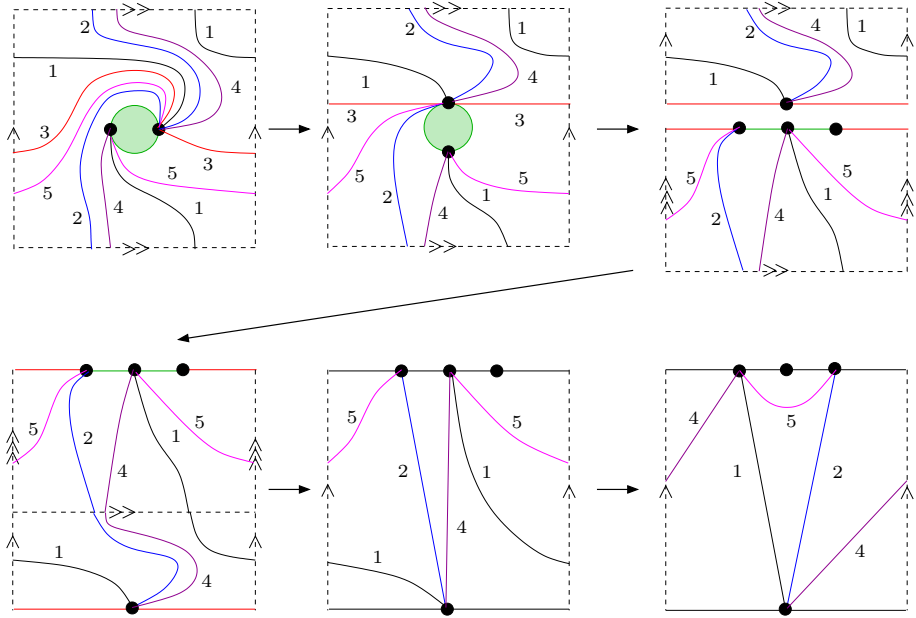} 
\end{center}
\caption{Cutting along an arc, numbered $3$, in a torus to get a cylinder:
triangulation case.}
\label{figure:qpexample1}
\end{figure}

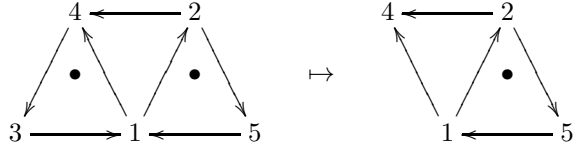
\begin{figure}
$$
\xymatrix@=0.4cm{
& 4 \ar[ddl] && 2\ar[ll] \ar[ddr] &&& 4 && 2 \ar[ddr] \ar[ll] & \\
& \bullet && \bullet && \mapsto &&& \bullet & \\
3 \ar[rr] && 1 \ar[uul] \ar[uur] && 5 \ar[ll] &&& 1 \ar[uul] \ar[uur] && 5 \ar[ll]
}
$$
\caption{The change in the quiver with potential from the cut in Figure~\ref{figure:qpexample1}. The potential in each case is given by the sum of
the $3$-cycles containing black dots.}
\label{figure:qpexample2}
\end{figure}

\begin{prop}
Let $(S,M)$ be a marked surface and $\arcs$ a partial triangulation of $(S,M)$.
Let $\arcs'$ be a collection of arcs containing $\arcs$. Then the coloured quiver
associated to $\pi_{\arcs}(\arcs'\setminus \arcs)$ in $(S,M)/\arcs$ coincides with the coloured quiver associated to $\arcs'$ in $(S,M)$ with the vertices corresponding to $\arcs$ and all arrows incident with them removed.
\end{prop}

\begin{proof}
It is clear that the vertices of each coloured quiver correspond to the arcs in
$\arcs'\setminus \arcs$. In the definition of the \twist\ $\kappa_{\arcs}$
(see Section~\ref{subsection: mutating arcs}), no distinction is made between
arcs in $\arcs$ and boundary arcs. Then, looking at the definition of the coloured quiver of a partial triangulation (see Section~\ref{ssection: coloured quivers}) we see that the arrows between arcs in $\arcs'\setminus \arcs$ are the same when
considered in either coloured quiver. The result follows.
\end{proof}

We now give an example.
In Figure~\ref{figure:cuttingexample1}, we start with a partial triangulation
of a torus with a single boundary component with two marked points. This has been
obtained by removing arcs $4$ and $5$ from the triangulation considered in
Figure~\ref{figure:qpexample1}. As before, we cut along the red arc (numbered $3$)
and obtain a cylinder with four marked points as shown, with a partial triangulation given by the remaining arcs. Figure~\ref{figure:cuttingexample2} gives the
corresponding coloured quiver associated to the partial triangulation in Figure~\ref{figure:cuttingexample1}, together with the new quiver obtained after cutting along the red arc (numbered $3$), i.e.\ with vertex $3$ and all arrows
incident with it removed.

\begin{figure}
\begin{center}
\psfragscanon
\psfrag{1}{$\scriptstyle 1$}
\psfrag{2}{\blue{$\scriptstyle 2$}}
\psfrag{3}{\red{$\scriptstyle 3$}}
\includegraphics[scale=0.5]{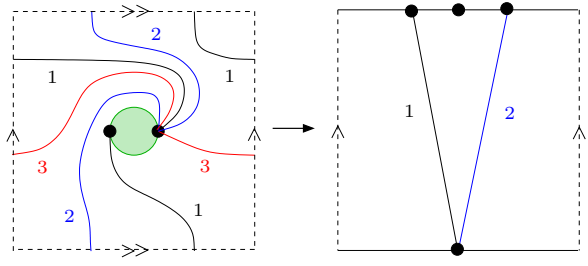} 
\end{center}
\caption{Cutting along an arc in a torus to get a cylinder: partial triangulation case.}
\label{figure:cuttingexample1}
\end{figure}

\begin{figure}
\begin{center}
\psfragscanon
\psfrag{(0,0)}{$\scriptstyle (0,0)$}
\psfrag{(0,1)}{$\scriptstyle (0,1)$}
\psfrag{(1,2)}{$\scriptstyle (1,2)$}
\psfrag{(2,2)}{$\scriptstyle (2,2)$}
\psfrag{3}{$\scriptstyle 3$}
\psfrag{B1}{\pscirclebox{$1$}} 
\psfrag{B2}{\pscirclebox{\blue{$2$}}} 
\psfrag{B3}{\pscirclebox{\red{$3$}}} 
\includegraphics[scale=0.4]{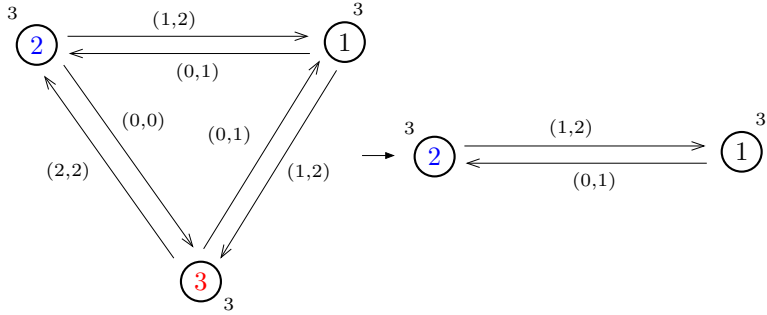} 
\end{center}
\caption{The effect of cutting along arc $3$ on the coloured quiver of the
partial triangulation in Figure~\ref{figure:cuttingexample1}.
The vertices themselves are encircled, with the vertex labels written outside
the circle.}
\label{figure:cuttingexample2}
\end{figure}

\section{Compatibility}

\subsection{Compatibility of the mutations}

Let $\arcs=\{\gamma_1,\ldots ,\gamma_m\}$ be a partial triangulation of $(S,M)$.
Complete $\arcs$ to a triangulation $\tri$ of $(S,M)$,
and let $T$ be the associated cluster tilting object in
$\cat=\cat_{(S,M)}$. Let $R$ be the direct summand of $T$
corresponding to $\arcs$.
We thus obtain a map $\al \mapsto X_\al$ between the arcs of
$(S,M)$ and the isomorphism classes of exceptional objects
in $\cat_{(S,M)}$ (see section~\ref{ssection: cluster surfaces}).
We denote by $\pi_{\arcs}$ the bijection
$A^0_\arcs(S,M) \fl A^0\big( (S,M)/\arcs \big)$;
recall also that $\pi_R$ denotes the functor $^\perp(\shift R) \fl \cat_R$.
Consider the partial triangulation $\pi_{\arcs}(\tri\setminus\arcs)$
of $(S,M)/\arcs$.
Note that $T':=\pi_R(T)$ is cluster tilting in $\cat_{(S,M)/\arcs}\simeq \cat_R$
by~\cite[Theorem 4.9]{IY}.
This cluster tilting object induces a bijection
$\beta \mapsto Y_\beta$ between the
arcs in $(S,M)/\arcs$ and the exceptional objects in $\cat_R$.

\begin{lemma} \label{lemma: reductioncorrespondence}
Let $\al$ be an arc in $A^0_\arcs(S,M)$.
Then the image of $X_\al$ under $\pi_R$
is isomorphic to $Y_{\pi_{\arcs}\al}$.
\end{lemma}

\begin{proof}
Using~\cite[Proposition 3.5]{JP}, the modules
associated with $\pi_R X_\al$ and
$Y_{\pi_{\arcs}\al}$ are seen to be isomorphic.
\end{proof}

Let $\al$ be an arc in $(S,M)$ which is not in $\arcs$ and
which does not cross $\arcs$, i.e. $\al\in A^0_{\arcs}(S,M)$.
Fix an orientation $\underline{\al}$ of $\al$ and let
$\al_s$ and $\al_t$ be the two (possibly boundary) arcs
following $\al$ in $\arcs$
(as defined in section~\ref{subsection: mutating arcs}).
Recall that $\comp_{\arcs}(\al)$ is defined to be
$[\underline{\al_t} \underline{\al} \underline{\al_s}^{-1}]$.

If $\gamma$ is any arc in $(S,M)$ then recall we have (from~\cite{BZ};
see Section~\ref{ssection: cluster surfaces}):
\begin{equation}
\label{equation: shift description}
\shift X_{\gamma}=X_{\comp_{\phi}(\gamma)},
\end{equation}
where $\phi$ denotes the empty set of arcs in $(S,M)$.
Thus $\comp_{\phi}(\al)$ is obtained from the arc $\al$ by composition with
the two boundary arcs which follow $\al$ (see
Section~\ref{subsection: mutating arcs}).

The following corollary describes the \twist\ of an arc
in terms of the action of the shift functor of an Iyama-Yoshino reduction.

\begin{cor}\label{corollary: mutation}
Under the bijection $A^0_\arcs(S,M) \leftrightarrow
A^0\big((S,M)/\arcs\big)$,
the induced action of the shift functor of $\cat_{(S,M)/\arcs}$
on $A^0_\arcs(S,M)$
coincides with that of the \twist\ $\comp_\arcs$.
In other words, we have a commutative diagram:
$$
\xymatrix{
A^0\big((S,M)/\arcs\big)
 \dr^{\text{shift}}                     &
A^0\big((S,M)/\arcs\big)      \\
A^0_\arcs(S,M) \ar[u]_{\pi_{\arcs}}  \dr^{\comp_\arcs}            &
A^0_\arcs(S,M). \ar[u]_{\pi_{\arcs}}
}
$$
\end{cor}

\begin{proof}
Let $\al\in A^0_{\arcs}(S,M)$. By~\eqref{equation: shift description},
noting that $\arcs$ becomes part of the boundary of $(S,M)/\arcs$,
we have
$\shift_R Y_{\pi_{\arcs}(\al)} \simeq Y_{\pi_{\arcs}\comp_{\arcs}(\al)}$.
The result follows.
\end{proof}

We now have the ingredients we need in order to show that the two mutations
(of partial triangulations and rigid objects) are compatible.

\begin{prop}\label{proposition: mutations coincide}
Let $R=R_1\oplus \cdots \oplus R_m$ be the rigid object in $\cat_{(S,M)}$
associated with the partial triangulation $\arcs$ as above.
Fix $1\leq k\leq m$. Then we have the indecomposable summand $R_k$ of $R$
and corresponding arc $\gamma_k$ of $\arcs$.
Let $\al$ be an arc in $A^0_{\arcs}(S,M)$. Then we have:
$$
\comp_R X_\al \simeq X_{\comp_\arcs(\al)}
\text{ and }
\mut_k R \simeq X_{\mut_k \arcs}.
$$
Hence, in particular, $d_k(R)=d_k(\arcs)$.
\end{prop}

\begin{proof}
Since $\al$ does not cross $\arcs$, it follows from~\cite{BZ}
(see Section~\ref{ssection: cluster surfaces}) that
$X_{\al}\in {}^{\perp}(\shift R)$. Similarly, $X_{\comp_{\arcs}(\al)}\in
{}^{\perp}(\shift R)$, since $\comp_{\arcs}(\al)$ does not cross $\arcs$.
By the description of the shift $\shift_R$ of $\cat_R$
in~\cite[4.1]{IY}, $\pi_R(\comp_R(X_{\al}))\simeq \shift_R(\pi_R X_{\al})$ in
$\cat_R$. By Lemma~\ref{lemma: reductioncorrespondence}, $\shift_R(\pi_R X_{\al})
\simeq \shift_R Y_{\pi_{\arcs}(\al)}$.
By Corollary~\ref{corollary: mutation}, we have $\shift_R Y_{\pi_{\arcs}(\al)}\simeq Y_{\pi_{\arcs}\comp_{\arcs}(\al)}$.
By Lemma~\ref{lemma: reductioncorrespondence},
$Y_{\pi_{\arcs}\comp_{\arcs}(\al)}\simeq \pi_R X_{\comp_\arcs}(\al)$. Hence
$\pi_R(\comp_R X_{\al})\simeq \pi_R(X_{\comp_{\arcs}(\al)})$.

Note that $\comp_R X_{\al}$ is an indecomposable object in $^{\perp}(\shift R)$
which is not in $\add R$ (see
Section~\ref{subsection: coloured quivers for rigid objects}).
Since $\comp_{\arcs}(\al)$ does not cross $\arcs$ and does not lie in $\arcs$,
the same is true of $X_{\comp_{\arcs}(\al)}$.
It follows that $\comp_R X_{\al}\simeq X_{\comp_{\arcs}(\al)}$,
proving the first part of the Proposition. The second and third statements
follow.
\end{proof}

\subsection{Compatibility of the coloured quivers}

As in the previous section, let $\al\in A^0_{\arcs}(S,M)$;
we fix an orientation of $\al$ and let $\al_s$ and $\al_t$
be the two (possibly boundary) arcs following $\al$ in $\arcs$
(as defined in section~\ref{subsection: mutating arcs}). Note
that it is possible that $\al_s=\al_t$.
We choose a triangulation $\tri$ of $(S,M)$ containing $\arcs$ and
$\alpha$. Let $T$ be the corresponding cluster tilting object,
containing $R$ as a direct summand and $X_{\alpha}$ as an
indecomposable direct summand.
Recall that $X_\gamma = 0$ if $\gamma$ is a boundary arc.

\begin{lemma}\label{lemma: approx}
 There is a minimal left $\add R$-approximation of $X_\al$ in $\cat_{(S,M)}$
 of the form $$X_\al \gfl X_{\al_s} \oplus X_{\al_t}.$$
\end{lemma}

\begin{proof}
By the $2$-Calabi-Yau tilting theorem (see
Section~\ref{ssection: generalised cluster categories}), the functor
$H=\cat(T,\Sigma\, -)$ induces an equivalence between $\cat/T$ and
$\operatorname{mod}J(Q,W)$. Hence $H$ induces an equivalence
between $\Sigma^{-1}\add T$ and the category $\mathcal{P}$ of projective modules
over $J(Q,W)$. Let $P_{\alpha}=H(\Sigma^{-1}X_{\alpha})$ for each arc $\alpha$
in $\tri$ and let $\mathcal{P}_R=H(\Sigma^{-1}\add R)$. Then it is enough to show that there is a minimal left $\mathcal{P}_R$-approximation of $P_{\al}$ in
$\operatorname{mod}J(Q,W)$ of the form
$$P_\al \gfl P_{\al_s} \oplus P_{\al_t}.$$

We recall that $J(Q,W)$ is gentle (see Section~\ref{ssection: Riemann surfaces}).
In particular, the defining relations are all zero-relations.
Let $\delta_1,\delta_2,\ldots ,\delta_j$ be the arcs in $\tri$ incident with
$\alpha(0)$ which are after $\alpha$ in the order induced by the orientation of the boundary at $\alpha(0)$ (and listed in that order); see
Section~\ref{subsection: mutating arcs}. Similarly, let
$\varepsilon_1,\varepsilon_2,\ldots ,\varepsilon_k$ be the arcs in $\tri$ around $\alpha(1)$
which are after $\alpha$ in the order induced by the orientation of the boundary at $\alpha(1)$.

Because of the zero-relations in $J(Q,W)$, the only non-zero paths in
$Q$ starting at $\alpha$ are paths:
$$
\alpha \gfl \delta_1 \gfl \delta_2 \gfl \cdots \gfl \delta_j
$$
and
$$
\alpha \gfl \varepsilon_1 \gfl \varepsilon_2 \gfl \cdots \gfl \varepsilon_k.
$$
Thus the only non-zero morphisms from $P_{\alpha}$ to some indecomposable
projective module lie in the composition chains:
$$
P_{\alpha} \gfl P_{\delta_1} \gfl P_{\delta_2} \gfl \cdots \gfl P_{\delta_j}
$$
and
$$
P_{\alpha} \gfl P_{\varepsilon_1} \gfl P_{\varepsilon_2} \gfl \cdots \gfl P_{\varepsilon_k},
$$
or are linear combinations of these (noting that the chains may overlap).

If $\alpha_t$ is a boundary arc, but $\alpha_s$ is not, then $\alpha_s$ occurs
in the first chain above. It is easy to see that
the non-zero map $P_{\alpha} \gfl P_{\alpha_s}$ coming from the chain of
compositions is a left minimal $\mathcal{P}_R$-approximation and we are done. The
argument is similar if $\alpha_s$ is a boundary arc but $\alpha_t$ is not. If both
$\alpha_s$ and $\alpha_t$ are boundary arcs then the zero map is a left minimal $\mathcal{P}_R$-approximation.

We are left with the case where neither $\alpha_s$ nor $\alpha_t$ is a boundary arc.
Thus $\alpha_s=\delta_i$ for some $i$ while $\alpha_t=\varepsilon_{i'}$ for
some $i'$. Let $f_s$ and $f_t$ be the non-zero morphisms arising from the above chains of compositions and let $f:P_{\alpha} \gfl P_{\alpha_s} \oplus P_{\alpha_t}$
be the map with components $f_s,f_t$. It follows from the above that $f$ is a left
$\mathcal{P}_R$-approximation of $P_{\alpha}$. It remains to check that $f$ is
left minimal.

We note that if we had $f_s=kh$ for some $h:P_{\alpha} \gfl P_{\beta}$ and
$k:P_{\beta} \gfl P_{\alpha_s}$ for some $\beta\in\arcs$ then $k$ would have to be
an isomorphism since the path in $Q$ from $\alpha$ to $\alpha_s$ is not equal
to any other path in $Q$ from $\alpha$ to $\alpha_s$, and $\alpha_s$ is the first arc in $\arcs$ appearing along this path. A similar statement holds for $f_t$.

If $f$ were not left minimal, a summand of form $0 \gfl P_{\alpha_s}$
(respectively, $0 \gfl P_{\alpha_t}$) would split off and we would have a
left $\mathcal{P}_R$-approximation of the form $g_s:P_{\alpha} \gfl P_{\alpha_s}$
(respectively, $g_t:P_{\alpha} \gfl P_{\alpha_t}$).
We consider only the first case (the second case requires a similar argument).
In this case, $f_t$ factors through $g_s$, i.e.\ $f_t=v g_s$ for
some map $v:P_{\alpha_s}\rightarrow P_{\alpha_t}$. By the above, $v$ is
an isomorphism and $g_s=v^{-1}f_t$.

Again, since $g_s$ is a left $\mathcal{P}_R$-approximation, we also have that $f_s$ factors through $g_s$, i.e. $f_s=w g_s$ for some
$w:P_{\alpha_s}\gfl P_{\alpha_t}$. By the above, $w$ is an isomorphism.
Hence we have $f_s=wv^{-1}f_t$ where $wv^{-1}$ is an isomorphism.
This is a contradiction since $f_s$ and $f_t$ arise from two different paths
starting at $\alpha$. The result is proved.
\end{proof}

\begin{theo}\label{Theorem: coloured quivers}
Let $R$ be the rigid object in $\cat_{(S,M)}$
associated with the partial triangulation $\arcs$.
Then the coloured quivers $Q_\arcs$ and $Q_R$ coincide.
\end{theo}

\begin{proof}
By Proposition~\ref{proposition: mutations coincide},
it is enough to prove that the sets of $0$-coloured
arrows coincide. This follows from Lemma~\ref{lemma: approx}.
\end{proof}

\section{Some examples}

\subsection{The $A_n$ case}

In this section, we assume that the category $\cat$ is the cluster category
of type $A_n$.

Suppose that $R=R_1\oplus \cdots \oplus R_m$ is a basic rigid object in
$\cat$. In Section~\ref{subsection: coloured quivers for rigid objects} we have associated a coloured quiver $Q$ with $R$.
If $R_k$ is an indecomposable direct summand of $R$ then the rigid object
$\mut_k R$ also has a coloured quiver, $\widetilde{Q}$, associated with it, and we can ask if $\widetilde{Q}$
can be computed from $Q$. This is known in the $d$-cluster-tilting object case
of a $d+1$-Calabi-Yau category~\cite[Thm.\ 2.1]{BuanThomas} but is not known for a general rigid object. In Section~\ref{section: mutation} we will indicate some results in this direction with a categorical proof, but here we give a complete answer for the cluster category of type $A$ using a combinatorial (geometric) proof. In this case,
the corresponding surface is a disk with $n+3$ marked points
(see~\cite{CCS}), which we shall denote $(S,M)$; as usual, we denote by $\arcs$ the set of noncrossing arcs in $(S,M)$ corresponding to the indecomposable
direct summands of $R$, writing $\gamma_i$ for the arc corresponding to $R_i$.
We may assume that all arcs are straight lines. We have seen above that we can
compute $Q$ using $\arcs$ instead of $R$.

If $R$ is indecomposable, the corresponding coloured quiver is trivial
(a single vertex and no arrows) and there is nothing more to do.
We assume we are not in this case.

The complement in $(S,M)$ of $\arcs\setminus \{\gamma_i\}$
is a union of disks, including one, $D_i$, containing $\gamma_i$, with
a polygonal boundary.
Then, by its definition, $\kappa_{\arcs\setminus \{\gamma_i\}}$
has the effect of rotating each of the endpoints of $\gamma_i$ anticlockwise
one edge around the polygonal boundary of $D_i$.

If the boundary of $D_i$ is a polygon with an even number of sides
and $\gamma_i$ joins two opposite vertices on this boundary then
we say that $\gamma_i$ is \emph{symmetric}.

\begin{lemma} \label{l:symmetric}
The arc $\gamma_i$ is symmetric if and only if, for any vertex $j$
such that $i$ and $j$ are ends of a common arrow in $Q$,
there is a unique colour $c$ such that $q_{(c)}(i,j)\not=0$.
\end{lemma}

\begin{proof}
Let $D_i$ be the disk defined above.
If $\gamma_i$ is symmetric, then the minimum number of \twists\
with respect to $\arcs\setminus \gamma_i$ required to return $\gamma_i$
to itself is equal to half the number of sides of the boundary of $D_i$
and this has the effect of rotating $\gamma_i$ through half a revolution.
It follows that if $i$ and $j$ are ends of a common arrow in $Q$,
there is a unique colour $c$ such that $q_{(c)}(i,j)\not=0$.
If $\gamma_i$ is not symmetric, the minimum number of \twists\ required
is equal to the total number of sides of the boundary of $D_i$, and this
has the effect of rotating $\gamma_i$ through a full revolution. We
see that if $i$ and $j$ are ends of a common arrow in $Q$,
there are exactly two distinct colours $c$ such that $q_{(c)}(i,j)\not=0$.
\end{proof}

We remark that a vertex $j$ as in Lemma~\ref{l:symmetric} always exists,
since we have assumed that $\arcs$ has more than one element.
It follows that whether $\gamma_i$ is symmetric or not is determined by
the coloured quiver $Q$. See Figure~\ref{figure:evencase} for an example.

\begin{figure}
\begin{center}
\psfragscanon
\psfrag{a}{$\scriptstyle\al$}
\psfrag{b}{$\scriptstyle\beta$}
\includegraphics[scale=0.4]{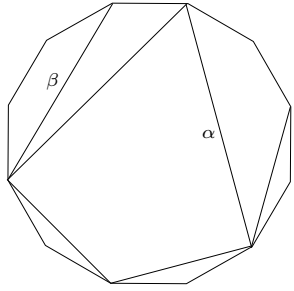} 
\end{center}
\caption{The arc $\al$ has order $3$ under \twisting\ and lies in a hexagon,
while the arc $\beta$ has order $5$ and lies in a pentagon.}
\label{figure:evencase}
\end{figure}

We have the following:

\begin{lemma} \label{lemma:typeAdiformula}
Fix $i\in \{1,2,\ldots ,m\}$.
Choose a vertex $j$ such that $q_{(c)}(i,j)\not=0$
for some $c$ (such a $j$ always exists by our assumption that $\arcs$
more than one element). Then we have:
$$d_i=\max\{c\geq 0\,:\, q_{(c)}(i,j)\not=0\}+\min\{c\geq 0\,:\,
q_{(c)}(j,i)\not=0\}+1.$$
\end{lemma}

\begin{proof}
Suppose first that $\gamma_i$ is not symmetric.
Then the situation is as in Figure~\ref{figure:diformula},
where the labels on the boundary indicate the number of edges along
sections of the boundary.
Since $d_i$ is the number of sides of $D_i$, we have that
$$d_i=(s+t)+r+1=\max\{c\geq 0\,:\, q_{(c)}(i,j)\not=0\}+\min\{c\geq 0\,:\,q_{(c)}(j,i)\not=0\}+1.$$
as required.
If $\gamma_i$ is symmetric, then $d_i$ is half the number of sides of $D_i$.
The situation can again be depicted as in Figure~\ref{figure:diformula}
(with the additional restriction that $t+r+1=s$) and we obtain:
$$d_i=t+r+1=\max\{c\geq 0\,:\, q_{(c)}(i,j)\not=0\}+\min\{c\geq 0\,:\,q_{(c)}(j,i)\not=0\}+1.$$
\end{proof}

\begin{figure}
\psfragscanon
\psfrag{r}{$r$}
\psfrag{s}{$s$}
\psfrag{t}{$t$}
\psfrag{i}{$i$}
\psfrag{j}{$j$}
\includegraphics[width=4cm]{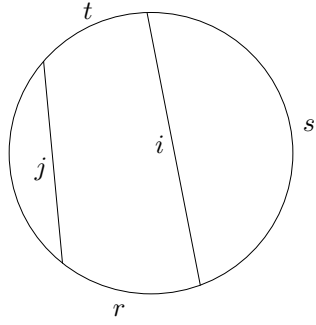}
\caption{Proof of Lemma~\ref{lemma:typeAdiformula}.}
\label{figure:diformula}
\end{figure}

The possibility of symmetric arcs makes it difficult to compute the new
coloured quiver after mutation of $\arcs$ at an arc, so we use a modified
version of the original quiver, defined as follows.

\vspace{.2cm}
\noindent \textbf{Definition}
The \emph{modified coloured quiver} $Q^+_{\arcs}$ of $\arcs$ is the
coloured quiver obtained from $Q_{\arcs}$ as follows. The vertices
are $\{1,2,\ldots ,m\}$. We set
$$
d^+_i(\arcs)=
\begin{cases}
2d_i(\arcs), & \text{if $\gamma_i$ is symmetric;} \\
d_i(\arcs), & \text{if $\gamma_i$ is not symmetric.}
\end{cases}
$$
If $\gamma_i$ is symmetric, we set $q^+_{(c)}(i,j)=q_{(c)}(i,j)$
and $q^+_{(c+d_i)}(i,j)=q_{(c)}(i,j)$ for all $c\in \{0,1,\ldots ,d_i-1\}$,
while if $\gamma_i$ is not symmetric, we set $q^+_{(c)}(i,j)=q_{(c)}(i,j)$
for all $c\in \{0,1,\ldots ,d_i-1\}$.
\vspace{.2cm}

Note that for any two distinct vertices $i$ and $j$, the modified coloured
quiver always has exactly two arrows from $i$ to $j$. We usually write this
as a single arrow labelled with the two colours $(l,l')$, with $l\leq l'$.
We also note that the arrows in $Q^+_{\arcs}$ can be obtained from $\arcs$
using the same rules (see Section~\ref{ssection: coloured quivers})
as for $Q_{\arcs}$ except that we use the numbers $d^+_i(\arcs)$ instead of the
numbers $d_i(\arcs)$. Using the same arguments as in the proof of
Lemma~\ref{lemma:typeAdiformula}, we have (again choosing a vertex $j$
such that $q^+_{(c)}(i,j)\not=0$ for some $c$):
\begin{equation}
\label{e:vertex}
d^+_i=\max\{c\geq 0\,:\, q^+_{(c)}(i,j)\not=0\}+\min\{c\geq 0\,:\,
q^+_{(c)}(j,i)\not=0\}+1.
\end{equation}

\begin{lemma} \label{lemma:modified}
The modified coloured quiver of $\arcs$ is determined by the coloured
quiver of $\arcs$ and vice versa.
\end{lemma}

\begin{proof}
Given the coloured quiver $Q_{\arcs}$ of $\arcs$, Lemma~\ref{l:symmetric}
indicates how to determine which arcs are symmetric, and thus how to
compute $Q^+_{\arcs}$ directly from $Q_{\arcs}$ using the definition above.
Note that an arc $\gamma_i$ is symmetric if and only if $d_i^+=d_i^+(\arcs)$ is
even and there is a vertex $j$ and a colour $c$ such that
$q^+_{(c)}(i,j)\not=0$ and $q^+_{(c+\frac{1}{2}d^+_i)}(i,j)\not=0$. It follows that
the coloured quiver of $\arcs$ can be determined from the modified coloured
quiver of $\arcs$.
\end{proof}

It is thus enough for us to give a method for determining the modified
coloured quiver of the mutation of a rigid object in terms of the modified
coloured quiver of the rigid object. We first compute the change in the
quiver in a number of cases.

Figures~\ref{figure:caseI}--\ref{figure:caseV}
each show a configuration of arcs in $(S,M)$,
together with the result after mutation at $\gamma_k$.
In each case, a label on part of the boundary indicates the number of
boundary edges between the two nearest arc ends on the boundary and
the black dot indicates the end of arc $k$ to show how this has changed
after the mutation. The following is a simple calculation:

\begin{lemma}
For each of the five cases in Figures~\ref{figure:caseI}--\ref{figure:caseV},
the effect of mutation at $k$ on the corresponding modified coloured quiver
is as shown.
\end{lemma}

\begin{figure}
\begin{center}
\psfragscanon
\psfrag{i}{$i$}
\psfrag{k}{$k$}
\psfrag{p}{$p$}
\psfrag{q-1}{$q-1$}
\psfrag{q}{$q$}
\psfrag{r+1}{$r+1$}
\psfrag{r}{$r$}
\psfrag{(r,r+p)}{$\scriptstyle (r,r+p)$}
\psfrag{(0,q)}{$\scriptstyle (0,q)$}
\psfrag{(0,p)}{$\scriptstyle (0,p)$}
\psfrag{(q-1,q+r)}{$\scriptstyle (q-1,q+r)$}
\psfrag{M}{$\mapsto$}
\includegraphics[scale=0.5]{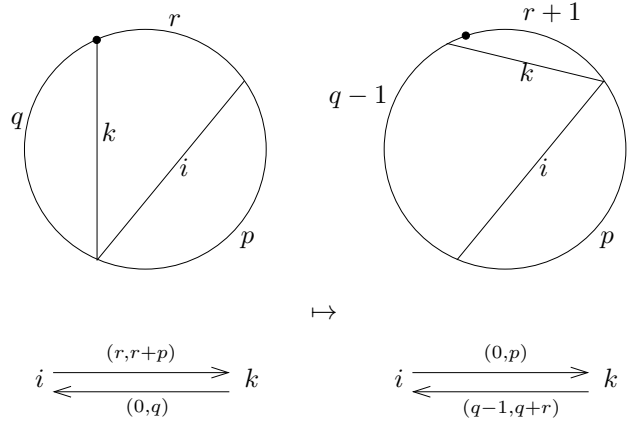} 
\end{center}
\caption{Case I. Here we have $p\geq 2$, $q\geq 2$, $r\geq 1$. Note that $d^+_i=p+r+1$, $\tilde{d}^+_i=p+q$ and $d^+_k=\tilde{d}^+_k=q+r+1$.}
\label{figure:caseI}
\end{figure}

\begin{figure}
\begin{center}
\psfragscanon
\psfrag{i}{$i$}
\psfrag{j}{$j$}
\psfrag{k}{$k$}
\psfrag{p}{$p$}
\psfrag{q-1}{$q-1$}
\psfrag{q}{$q$}
\psfrag{r+1}{$r+1$}
\psfrag{r}{$r$}
\psfrag{s}{$s$}
\psfrag{t}{$t$}
\psfrag{(t,p+t)}{$\scriptstyle (t,p+t)$}
\psfrag{(q,q+r+1)}{$\scriptstyle (q,q+r+1)$}
\psfrag{(r,r+s)}{$\scriptstyle (r,r+s)$}
\psfrag{(0,q+t+1)}{$\scriptstyle (0,q+t+1)$}
\psfrag{(q,q+s)}{$\scriptstyle (q,q+s)$}
\psfrag{(t+1,p+t+1)}{$\scriptstyle (t+1,p+t+1)$}
\psfrag{(q-1,q+r)}{$\scriptstyle (q-1,q+r)$}
\psfrag{(q+t,q+r+t+1)}{$\scriptstyle (q+t,q+r+t+1)$}
\psfrag{(0,s)}{$\scriptstyle (0,s)$}
\psfrag{M}{$\mapsto$}
\includegraphics[scale=0.5]{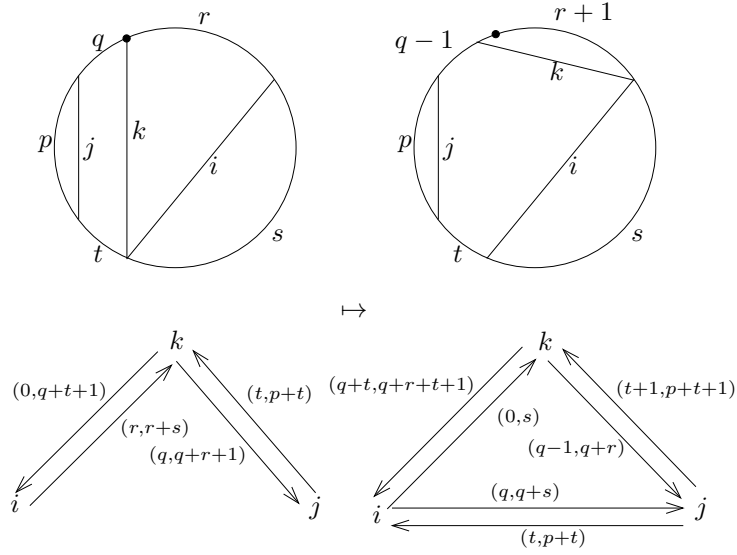} 
\end{center}
\caption{Case II. Here we have $p\geq 2$, $q\geq 1$, $r\geq 1$, $s\geq 2$, $t\geq 0$. Note that $d^+_i=r+s+1$, $\tilde{d}^+_i=q+t+s+1$,
$d^+_j=\tilde{d}^+_j=p+q+t+1$ and $d^+_k=\tilde{d}^+_k=q+r+t+2$.}
\label{figure:caseII}
\end{figure}

\begin{figure}
\begin{center}
\psfragscanon
\psfrag{i}{$i$}
\psfrag{j}{$j$}
\psfrag{k}{$k$}
\psfrag{p}{$p$}
\psfrag{q}{$q$}
\psfrag{r}{$r$}
\psfrag{s}{$s$}
\psfrag{s-1}{$s-1$}
\psfrag{t}{$t$}
\psfrag{t+1}{$t+1$}
\psfrag{(t,p+t)}{$\scriptstyle (t,p+t)$}
\psfrag{(q+1,q+s+1)}{$\scriptstyle (q+1,q+s+1)$}
\psfrag{(0,s)}{$\scriptstyle (0,s)$}
\psfrag{(q+t+1,q+r+t+1)}{$\scriptstyle (q+t+1,q+r+t+1)$}
\psfrag{(t+1,p+t+1)}{$\scriptstyle (t+1,p+t+1)$}
\psfrag{(q,q+r)}{$\scriptstyle (q,q+r)$}
\psfrag{(q,q+s)}{$\scriptstyle (q,q+s)$}
\psfrag{(s-1,q+s+t+1)}{$\scriptstyle (s-1,q+s+t+1)$}
\psfrag{(0,r)}{$\scriptstyle (0,r)$}
\psfrag{M}{$\mapsto$}
\includegraphics[scale=0.5]{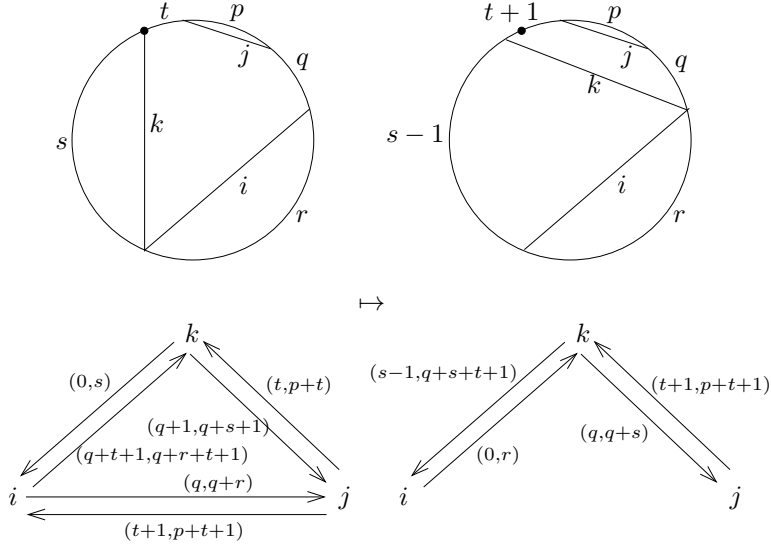} 
\end{center}
\caption{Case III. Here we have $p\geq 2$, $q\geq 0$, $r\geq 2$, $s\geq 2$, $t\geq 0$. Note that $d_i^+=q+r+t+2$, $\tilde{d}^+_i=r+s$,
$d_j^+=\tilde{d}_j^+=p+q+t+2$ and $d_k^+=\tilde{d}_k^+=q+s+t+2$.}
\label{figure:caseIII}
\end{figure}

\begin{figure}
\begin{center}
\psfragscanon
\psfrag{i}{$i$}
\psfrag{j}{$j$}
\psfrag{k}{$k$}
\psfrag{p}{$p$}
\psfrag{q}{$q$}
\psfrag{r}{$r$}
\psfrag{r-1}{$r-1$}
\psfrag{s}{$s$}
\psfrag{s+1}{$s+1$}
\psfrag{t}{$t$}
\psfrag{(t,p+t)}{$\scriptstyle (t,p+t)$}
\psfrag{(q,q+s+1)}{$\scriptstyle (q,q+s+1)$}
\psfrag{(0,r)}{$\scriptstyle (0,r)$}
\psfrag{(s,q+s+t+1)}{$\scriptstyle (s,q+s+t+1)$}
\psfrag{(r-1,r+s)}{$\scriptstyle (r-1,r+s)$}
\psfrag{(q,q+r)}{$\scriptstyle (q,q+r)$}
\psfrag{(0,q+t+1)}{$\scriptstyle (0,q+t+1)$}
\psfrag{M}{$\mapsto$}
\includegraphics[scale=0.5]{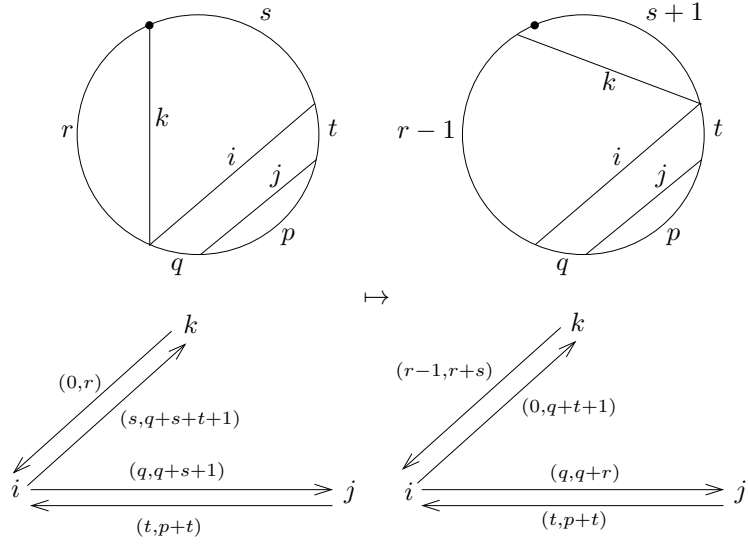} 
\end{center}
\caption{Case IV. Here we have $p\geq 2$, $q\geq 0$, $r\geq 2$, $s\geq 1$, $t\geq 0$, $q+t\geq 1$. Note that $d_i^+=q+s+t+2$, $\tilde{d}^+_i=q+r+t+1$,
$d_j^+=\tilde{d}_j^+=p+q+t+1$ and $d_k^+=\tilde{d}^+_k=r+s+1$.}
\label{figure:caseIV}
\end{figure}

\begin{figure}
\begin{center}
\psfragscanon
\psfrag{i1}{$i_1$}
\psfrag{i2}{$i_2$}
\psfrag{k}{$k$}
\psfrag{p}{$p$}
\psfrag{q}{$q$}
\psfrag{r}{$r$}
\psfrag{s}{$s$}
\psfrag{(q,q+s+1)}{$\scriptstyle (q,q+s+1)$}
\psfrag{(0,r)}{$\scriptstyle (0,r)$}
\psfrag{(0,p)}{$\scriptstyle (0,p)$}
\psfrag{(s,q+s+1)}{$\scriptstyle (s,q+s+1)$}
\psfrag{(s,p+s)}{$\scriptstyle (s,p+s)$}
\psfrag{(0,q+1)}{$\scriptstyle (0,q+1)$}
\psfrag{(0,s+1)}{$\scriptstyle (0,s+1)$}
\psfrag{(q,q+r)}{$\scriptstyle (q,q+r)$}
\psfrag{M}{$\mapsto$}
\includegraphics[scale=0.5]{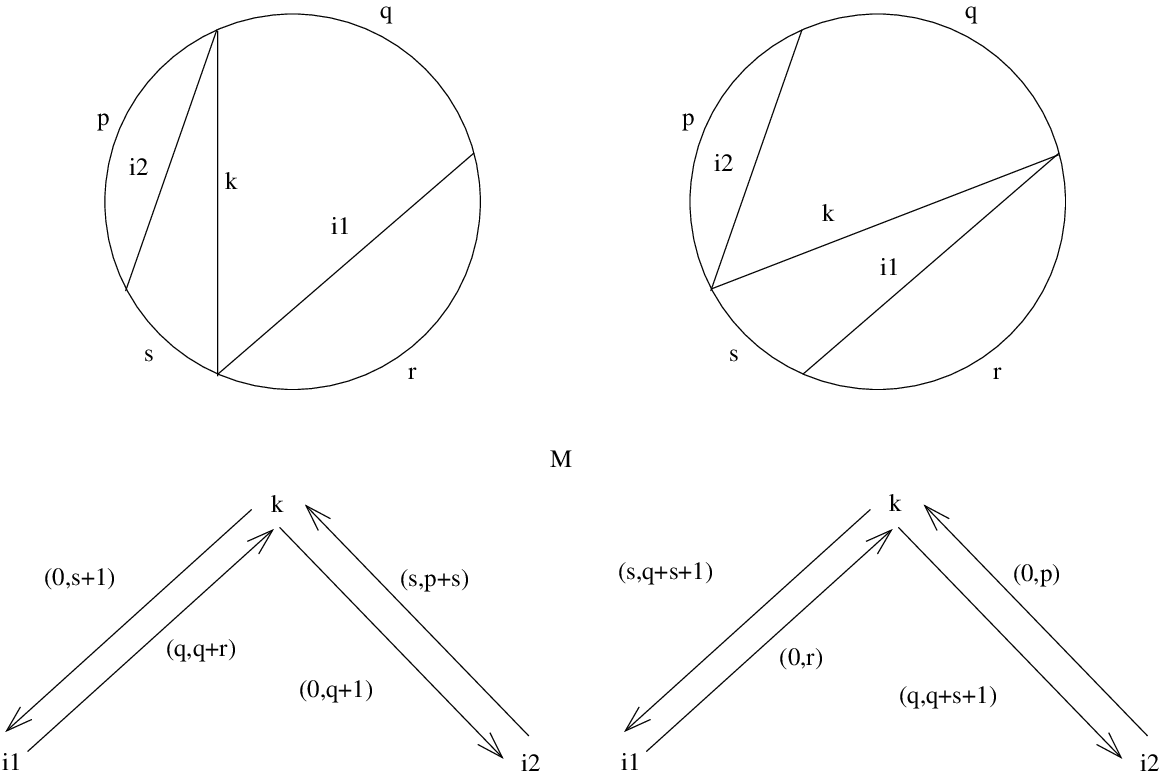} 
\end{center}
\caption{Case V. Here we have $p\geq 2$, $q\geq 1$, $r\geq 2$, $s\geq 1$.
Note that $d_{i_1}^+=q+r+1$, $\tilde{d}_{i_1}^+=r+s+1$, $d^+_{i_2}=p+s+1$,
$\tilde{d}^+_{i_2}=p+q+1$ and $d_k^+=\tilde{d}_k^+=q+s+2$.}
\label{figure:caseV}
\end{figure}

We next consider, in the general case, the effect of mutation at a
vertex $k$ on the modified coloured quiver $Q^+=Q^+_{\arcs}$.
Let $k^*$ be the set of vertices at the targets of arrows starting at
$k$ with colour zero. A vertex $i$ lies in $k^*$ if and only if $\gamma_i$
is incident with a single endpoint of $\gamma_k$ and the clockwise
angle from $\gamma_k$ and $\gamma_i$ is inside $(S,M)$; see the
diagram on the left hand side of Figure~\ref{figure:caseI}, where $i$
is an example of an element of $k^*$.

Suppose $i\not\in k^*$ and that there is an arrow from $i$ to $k$ in $Q^+$.
Then neither of the colours in the label of the arrow is equal to
$d^+_i-1$ (else $i$ would lie in $k^*$).
The effect of mutating at $k$ is to increase both of these colours by $1$.
If there is an arrow from $k$ to $i$ in $Q^+$ then, for the same reason,
the colours in the label of the arrow are non-zero and mutation at $k$ decreases these by $1$.

If $i\in k^*$, then the effect on $Q^+$ of mutating at $k$ is
shown in Figure~\ref{figure:caseI}: the disk in this picture should be
interpreted as the boundary of the connected component of the complement
in $(S,M)$ of $\arcs \setminus \{\gamma_i,\gamma_k\}$
containing $\gamma_i$ and $\gamma_k$.

Mutation at $k$ can only affect the vertex labels corresponding to arcs on the
boundary of the connected component of the complement in $(S,M)$ of the arcs
in $\arcs\setminus \{\gamma_k\}$ containing $\gamma_k$,
so no other vertex labels can change.

We observe that mutation at $k$ can only affect arrows in $Q^+$
incident with $k$ (already considered above) or incident with at least one
vertex in $k^*$.

Consider first the arrows between vertices $i\in k^*$ and $j\not\in k^*$.
Figures~\ref{figure:caseII},~\ref{figure:caseIII} and~\ref{figure:caseIV}
show the three possibilities for the location of $\gamma_j$. In each case
the disk should be interpreted as the boundary of the connected component
of the complement in $(S,M)$ of the arcs in
$\arcs\setminus \{\gamma_i,\gamma_j,\gamma_k\}$ containing $\gamma_i,\gamma_j$
and $\gamma_k$.

The effect of mutation on the corresponding full subquiver of $Q^+$ is as shown.

Finally, we consider the arrows between vertices $i_1,i_2\in k^*$.
This case is shown in Figure~\ref{figure:caseIV}, with the disk interpreted
as the boundary of the connected component of the complement in $(S,M)$ of
the arcs in $\arcs\setminus \{\gamma_{i_1},\gamma_{i_2},\gamma_k\}$ containing
$\gamma_{i_1},\gamma_{i_2}$ and $\gamma_k$.
The effect of mutation on the corresponding full subquiver of $Q^+$ is as shown.

Analysing the effect of mutation leads us to the following
method for mutating a (modified) coloured quiver in type A.

\begin{prop}
Let $R=R_1\oplus \cdots \oplus R_m$ be a rigid object in the cluster category of
type $A_n$, with associated modified coloured quiver $Q^+$.
Let $R_k$ be a summand of $R$ and let $\widetilde{Q}^+$ denote the modified
coloured quiver of $\mu_k(R)$, with periodicity $\tilde{d}_i$ associated
to vertex $i$. Then $\widetilde{Q}^+$ can be computed in the following way.
The letters $i,j,k$ always refer to distinct vertices.

\vskip 0.2cm
\noindent \textbf{Type A modified coloured quiver mutation:}
\vskip 0.2cm

\begin{enumerate}
\item[(i)]
Suppose we have the following arrows:
$$
\xymatrix@C=1.5cm{i \ar@<1ex>^{(a,a')}[r] & k \ar@<1ex>^{(0,b')}[l]
&
j \ar@<1ex>^{(c,c')}[r] & k \ar@<1ex>^{(d,d')}[l]
}
$$
where $d\not=0$. Add the following arrows:
$$\xymatrix@C=1.5cm{i \ar@<1ex>^{(d,d+a'-a)}[r] & j \ar@<1ex>^{(c,c')}[l]}$$
and cancel any pairs of arrows between $i$ and $j$ in the same direction whose
colours differ by $1$.
\item[(ii)]
Suppose we have the following full subquiver of $Q$:
$$
\xymatrix@C=1.5cm{k \ar@<1ex>^{(0,b')}[r] & i \ar@<1ex>^{(a,a')}[l]
\ar@<1ex>^{(d,d')}[r] & j \ar@<1ex>^{(c,c')}[l]}.
$$
Then change the arrows between $i$ and $j$ to:
$$\xymatrix@C=1.5cm{i \ar@<1ex>^{(d,d+b')}[r] & j \ar@<1ex>^{(c,c')}[l]}.$$
\item[(iii)]
Apply the following rule to all vertices $i$ with an arrow to or from $k$:
$$
\xymatrix@C=1.5cm{i \ar@<1ex>^{(a,a')}[r] & k \ar@<1ex>^{(b,b')}[l]}
\mapsto
\begin{cases}
\xymatrix@C=1.5cm{i \ar@<1ex>^{(a+1,a'+1)}[r] & k \ar@<1ex>^{(b-1,b'-1)}[l]}
& \mbox{if\ }b\not=0; \\
\xymatrix@C=1.5cm{i \ar@<1ex>^{(0,a'-a)}[r] & k \ar@<1ex>^{(b'-1,a+b')}[l]}
 & \mbox{if\ }b=0.
\end{cases}
$$
If $b=0$, add $b'-a$ to the label at vertex $i$, giving new value $a'+b'-a$.
Otherwise, the vertex labels are unchanged.
\end{enumerate}
\end{prop}

\begin{proof}
By the discussion above, Step (iii) indicates how arrows incident with $k$
change under mutation. The case $b=0$ is shown in Case I (see
Figure~\ref{figure:caseI}), where we take
$a=r$, $a'=r+p$, $b=0$ and $b'=q$ and we note that it is easy
to check using~\eqref{e:vertex} that the change in vertex labels is as claimed.

Since mutation at $k$ can only affect these arrows or those arrows incident
with at least one vertex in $k^*$, it remains only to check that applying
the above method has the right effect in Cases II-V considered above.
Case II (Figure~\ref{figure:caseII}) can be regarded as an instance of Step (i) with
$a=r$, $a'=r+s$, $b=0$, $b'=q+t+1$, $c=t$, $c'=p+t$, $d=q$ and $d'=q+r+1$,
followed by Step (iii).
Case III (Figure~\ref{figure:caseIII}) can be regarded as an instance of Step (i) with
$a=q+t+1$, $a'=q+r+t+1$, $b=0$, $b'=s$, $c=t$, $c'=p+t$, $d=q+1$ and $d'=q+s+1$, followed by Step (iii).
Case IV (Figure~\ref{figure:caseIV}) can be regarded as an instance of Step (ii) with $a=s$,
$a'=q+s+t+1$, $b=0$, $b'=r$, $c=t$, $c'=p+t$, $d=q$ and $d'=q+s+1$, followed by Step (iii).
In Case V (Figure~\ref{figure:caseV}) we see that only Step (iii) is applied, and we are done.
\end{proof}

\subsection{An example with infinitely many colours}
\label{ssection: infinitecolour}

We consider again the example from Figure~\ref{figure:cuttingexample1},
i.e.\ a torus with a single boundary component with two marked points.
We show again the partial triangulation of this surface in
Figure~\ref{figure:infiniteexample1}.
Several copies are drawn to make it easier to see mutations at each of the arcs.
The corresponding coloured quiver is given below the surface: note
that removing any of the three arcs leaves a hexagon; it follows
that mutation at any of the arcs has order $3$, and we get finitely
many colours: $0$, $1$ and $2$, appearing as labels on the arrows.

Now suppose we mutate at arc $1$. We obtain the partial triangulation in
Figure~\ref{figure:infiniteexample2}; the corresponding quiver is
given below the picture of the surface. Here, an arrow is labelled
with $\mathbb{Z}$ to represent an infinite number of arrows, one coloured
$n$ for each integer $n$. This infinity of arrows comes from mutating
at arc number $2$. If we cut along the remaining arcs in the partial
triangulation, we obtain a cylinder.
Then, after each mutation a small neighbourhood of the triangulation
is the same at each end of the arc (which explains the regularity),
but as more and
more mutations are made the arc wraps itself more and more around
the cylinder. Thus we see that, even if the quiver is locally finite to start
with, after a mutation it might not be.

\begin{figure}
\begin{center}
\psfragscanon
\psfrag{1}{$\scriptstyle 1$}
\psfrag{2}{$\scriptstyle 2$}
\psfrag{3}{$\scriptstyle 3$}
\psfrag{(0,0)}{$\scriptstyle (0,0)$}
\psfrag{(0,1)}{$\scriptstyle (0,1)$}
\psfrag{(1,2)}{$\scriptstyle (1,2)$}
\psfrag{(2,2)}{$\scriptstyle (2,2)$}
\psfrag{B1}{\pscirclebox{$1$}} 
\psfrag{B2}{\pscirclebox{\blue{$2$}}} 
\psfrag{B3}{\pscirclebox{\red{$3$}}} 
\includegraphics[scale=0.4]{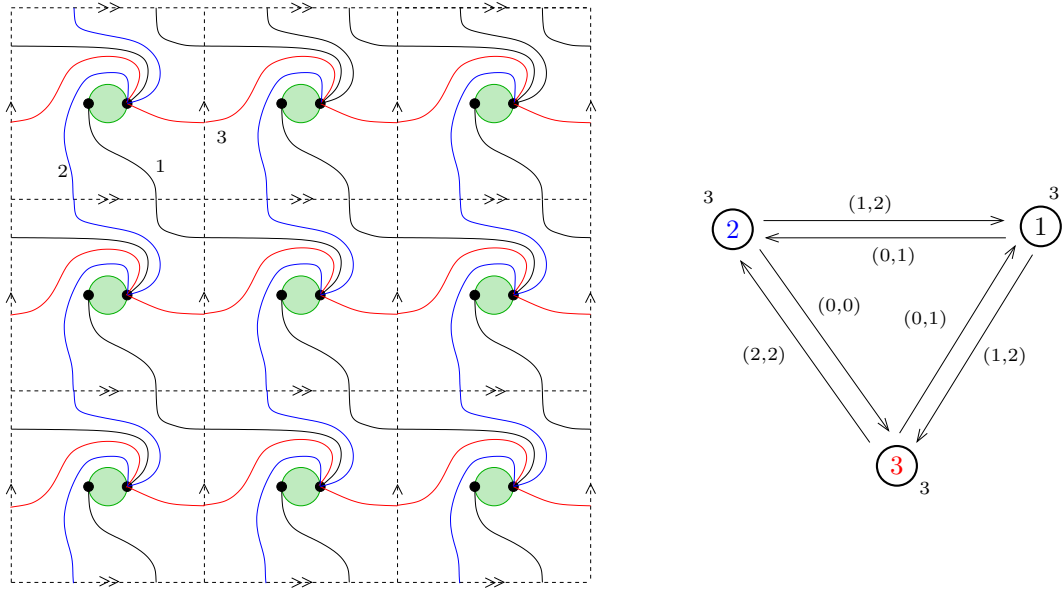} 
\end{center}
\caption{A partial triangulation of a torus with a single boundary component and two marked points and the corresponding coloured quiver.
}
\label{figure:infiniteexample1}
\end{figure}

\begin{figure}
\begin{center}
\psfragscanon
\psfrag{0}{$\scriptstyle 0$}
\psfrag{1}{$\scriptstyle 1$}
\psfrag{2}{$\scriptstyle 2$}
\psfrag{3}{$\scriptstyle 3$}
\psfrag{(0,1)}{$\scriptstyle (0,1)$}
\psfrag{(1,2)}{$\scriptstyle (1,2)$}
\psfrag{(0,2)}{$\scriptstyle (0,2)$}
\psfrag{B1}{\pscirclebox{$1$}} 
\psfrag{B2}{\pscirclebox{\blue{$2$}}} 
\psfrag{B3}{\pscirclebox{\red{$3$}}} 
\psfrag{Z}{$\scriptstyle \mathbb{Z}$}
\includegraphics[scale=0.4]{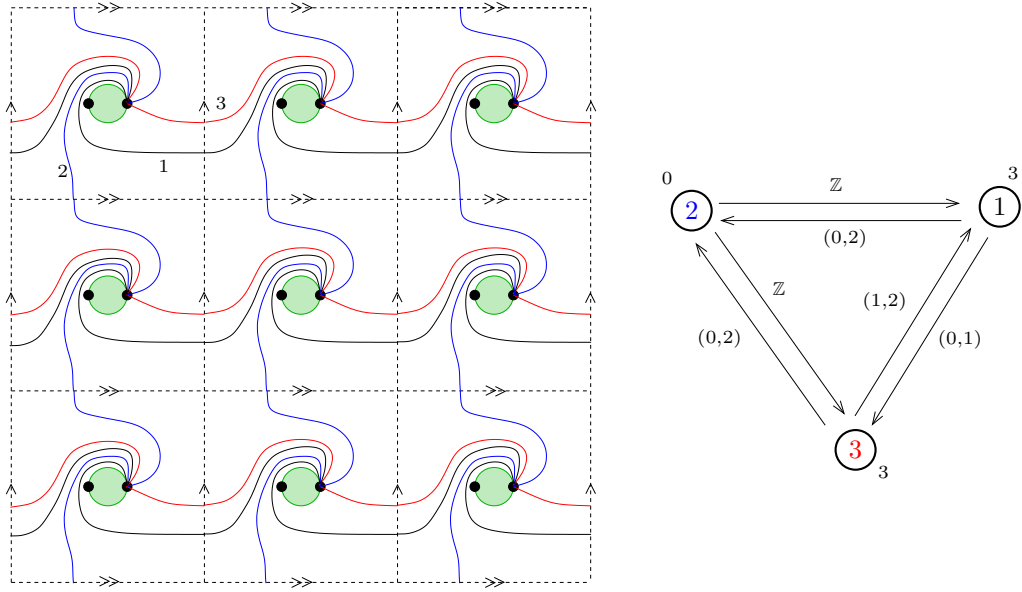} 
\end{center}
\caption{The result of mutating the partial triangulation in Figure~\ref{figure:infiniteexample1} at arc $1$ and the corresponding coloured quiver.}
\label{figure:infiniteexample2}
\end{figure}

\begin{rk}
We note that in the example in Figure~\ref{figure:infiniteexample2},
the coloured quiver contains
a two-cycle of arrows both coloured zero, a situation that does not arise in the
coloured quivers arising in $m$-cluster categories~\cite[Sect.\ 2]{BuanThomas}.

\end{rk}


\section{Partial categorical interpretation}\label{section: mutation}

In this section, we prove the following result, which gives
a partial description of the mutation of a coloured quiver associated
to a rigid object in a categorical context. Note that the result
applies to any $2$-Calabi-Yau triangulated category (under mild
assumptions), and is not restricted to the surface case.

\begin{theo}\label{Theorem: mutation}
Let $\cat$ be a Hom-finite,
Krull--Schmidt, 2-Calabi--Yau triangulated $K$-category.
Let $Q$ be the coloured quiver associated with
a rigid object $R=R_1\oplus \cdots \oplus R_m\in\cat$ and
let $\widetilde{Q}$ be the coloured quiver
associated with $\mu_k R$, for some vertex $k$ of $Q$.
Denote the periodicity associated with vertex $i$
of $Q$ (resp. $\widetilde{Q}$) by $d_i$ (resp. $\tilde{d}_i$).
\begin{itemize}
 \item[(i)] We have:
 \begin{itemize}
  \item[$\bullet$] $d_k = \tilde{d}_k$ and
  \item[$\bullet$] for any $j\in Q_0$ and any $c\in\zb/d_k$,
  $\widetilde{q}_{(c)}(k,j) = q_{(c+1)}(k,j)$.
 \end{itemize}
 \item[(ii)] Let $i,j\in Q_0$ be such that
 $q_{(0)}(k,j)=0=q_{(0)}(k,i)$. Then we have:
 \begin{itemize}
  \item[$\bullet$] $\tilde{d}_i=d_i$, $\tilde{d}_j=d_j$;
  \item[$\bullet$] for any $c\in\zb/d_j$, $\widetilde{q}_{(c)}(j,k) = q_{(c-1)}(j,k)$;
  \item[$\bullet$] for any $c\in\zb/d_i$, $\widetilde{q}_{(c)}(i,j) = q_{(c)}(i,j)$.
 \end{itemize}
\end{itemize}
\end{theo}

We note that, for the cases covered by the theorem, the new coloured
quiver depends only on the old coloured quiver and not on the particular
choice of rigid object or category $\cat$.

\subsection{Proof of Theorem~\ref{Theorem: mutation}}

We break the proof of Theorem~\ref{Theorem: mutation}
down into smaller steps, which we present as individual
lemmas.

Let $R=R_1\oplus \cdots \oplus R_m\in\cat$ be rigid and let $Q$ be the
associated coloured quiver.

\begin{lemma}\label{lemma: from k}
 We have $\tilde{d}_k=d_k$ and, for any $c\in\zb/d_k$ and
 any $j\in Q_0$,
 $$
 \widetilde{q}_{(c)}(k,j) = q_{(c+1)}(k,j).
 $$
\end{lemma}

\begin{proof}
 The exchange triangles for $R_k^{(1)}$
 can be deduced from those for $R_k$, so that
 we have
 $(R_k^{(1)})^{(c)} = R_k^{(c+1)}$.
\end{proof}

\begin{lemma}\label{lemma: to k easy case}
 Let $j\in Q_0$ be such that $q_{(0)}(k,j) = 0$.
 Then $\tilde{d}_j = d_j$ and for any $c\in\zb/d_j$ we have:
 $$ \widetilde{q}_{(c)}(j,k) = q_{(c-1)}(j,k).$$
\end{lemma}

\begin{proof}
 By Corollary~\ref{corollary: reduction},
 we may assume that $R=R_j \oplus R_k$.
 Since $q_{(0)}(k,j) = 0$, the first exchange
 triangle for $R_k$ with respect to $R_j$ is:
 $$
 R_k \gfl 0 \gfl R_k^{(1)} \stackrel{=}{\gfl} \shift R_k.
 $$
 Let
 $$
 \ldots,\;
 R_j^{(-1)} \gfl R_k^{s_{-1}} \gfl R_j \gfl \shift R_j^{(-1)} \;,\;
 R_j \gfl R_k^{s_0} \gfl R_j^{(1)} \gfl \shift R_j \;,\ldots
 $$
be the exchange triangles for $R_j$ with respect to $R_k$.
Since $q_{(0)}(k,j) = 0$, we have $s_{-1} =0$
and $R_j$ is isomorphic to $\shift R_j^{(-1)}$.
The exchange triangles for $R_j$ with respect to
$R_k^{(1)} = \shift R_k$ are thus obtained from those
with respect to $R_k$ by applying the shift functor:
\begin{itemize}
 \item[] \hspace{2cm} $\vdots$
 \item[] $\shift R_j^{(-3)} \gfl \shift R_k^{s_{-3}} \gfl \shift R_j^{(-2)} \gfl$
 \item[] $\shift R_j^{(-2)} \gfl \shift R_k^{s_{-2}} \gfl
 \phantom{\shift}R_j^{\phantom{(-2)}} \gfl$
 \item[] $\phantom{\shift}R_j^{\phantom{(2)}} \gfl
 \phantom{\shift} 0^{\phantom{s_0}}\hspace{2pt} \gfl \shift R_j^{\phantom{(2)}} \gfl$
 \item[] $\shift R_j^{\phantom{(2)}} \gfl \shift R_k^{s_0} \gfl \shift R_j^{(1)} \gfl$
 \item[] $\shift R_j^{(1)} \gfl \shift R_k^{s_1} \gfl \shift R_j^{(2)} \gfl$
 \item[] \hspace{2cm} $\vdots$
\end{itemize}
\end{proof}

\begin{lemma}\label{lemma: 3 vertices easy case}
 Let $i,j\in Q_0$ be such that $q_{(0)}(k,j) = 0 = q_{(0)}(k,i)$.
 Then, for any $c\in\zb/d_i$, we have:
 $$\widetilde{q}_{(c)}(i,j) = q_{(c)}(i,j).$$
\end{lemma}

\begin{proof}
 By Corollary~\ref{corollary: reduction},
 we may assume that $R=R_i\oplus R_j \oplus R_k$.
 Let
 $$
 \ldots,\;
 R_i^{(-1)} \fl R_j^{t_{-1}} \fl R_i \fl \shift R_i^{(-1)}, \;
 R_i \fl R_k^{s_0}\oplus R_j^{t_0} \fl R_i^{(1)} \fl \shift R_i
 \;,\ldots
 $$
 be the exchange triangles in $\cat$ for $R_i$ with respect to $R_j\oplus R_k$.
 We denote by $R_i^{(c)\ast}$ the \twists\ of $R_i$
 with respect to $\mu_kR/R_i$.
 Our assumptions have the following consequences:
 \begin{itemize}
  \item[(i)] $R_k^{(1)} = \shift R_k$ and
  \item[(ii)] the spaces $\cat(R_k,R_j)$ and $\cat(R_k,R_i)$ vanish.
 \end{itemize}
 Let $\cs$ be the Iyama--Yoshino reduction of $\cat$
 with respect to $\shift R_k=R_k^{(1)}$. The image in $\cs$
 of a morphism $f\in\cat$ is denoted $\underline{f}$.

 By induction on $c\geq 0$, we are going to construct:
 \begin{itemize}
  \item[(a)] A minimal left $\add R_j$-approximation
  $R_i^{(c)\ast} \gfl R_j^{t_c}$ in $\cs$;
  \item[(b)] a triangle $X_{c+1} \fl R_i^{(c+1)} \fl R_i^{(c+1)\ast} \fl \shift X_{c+1}$
  in $\cat$, with $X_{c+1}$ in $\add R_k$;
  \item[(a')] a minimal right $\add R_j$-approximation
  $R_j^{t_{-c-1}} \gfl R_i^{(-c)\ast}$ in $\cs$ and
  \item[(b')] a triangle $X_{-c-1} \fl R_i^{(-c-1)} \fl R_i^{(-c-1)\ast} \fl \shift X_{-c-1}$
  in $\cat$, with $X_{-c-1}$ in $\add R_k$.
 \end{itemize}
 The result then follows from (a) and (a') by Corollary~\ref{corollary: reduction}.

 Let us first prove that (a) and (b) hold for $c=0$.
 Note that, by (ii), both $R_i$ and $R_j$ belong to
 $(\susm R^{(1)}_k)^\perp$, so that (a) makes sense.
 Let us denote by $\left[_f^{f'}\right]$ the map
 $R_i \fl R_k^{s_0}\oplus R_j^{t_0}$.
 Since $\cat(R_k,R_j)=0$, the map $R_i \stackrel{f}{\gfl} R_j^{t_0}$
 is a left $\add R_j$-approximation in $\cat$, thus so is
 $\underline{f}$ in $\cs$. Let $g\in\End_{\cat}(R_j^{t_0})$ be such
 that $\underline{g}\,\underline{f} = \underline{f}$. Then
 $\left[_{0 \; g}^{1 \; 0} \right]\left[_f^{f'}\right] - \left[_f^{f'}\right]$
 factors through $\add \shift R_k$. Since $R_i$ belongs to
 $^\perp(\shift R_k)$, we have in fact
 $\left[_{0 \; g}^{1 \; 0} \right]\left[_f^{f'}\right] = \left[_f^{f'}\right]$.
By minimality, $g$ is an isomorphism. Thus $\underline{f}$
is left-minimal.
By Lemma~\ref{lemma: CY approx} and Lemma~\ref{lemma: to k easy case},
the first exchange triangle with respect to $\mu_k R$ for $R_i$ is
$$
R_i \gfl R_j^{t_0} \gfl R_i^{(1)\ast} \gfl \shift R_i.
$$
The triangle (b) is easily constructed by applying the
octahedral axiom to the composition
$R_i \gfl R_k^{s_0}\oplus R_j^{t_0} \stackrel{\text{proj}}{\gfl} R_j^{t_0}$ as follows:
$$
\xymatrix{
                          &
R_k^{s_0} \dreg \bas      &
R_k^{s_0} \bas            &
                          \\
R_i \dr \baseg            &
R_k^{s_0}\oplus
R_j^{t_0} \dr \bas        &
R_i^{(1)} \dr \bas        &
\shift R_i \baseg         \\
R_i \dr                   &
R_j^{t_0} \dr \bas_0      &
R_i^{(1)\ast} \dr \bas    &
\shift R_i                \\
                          &
\shift R_k^{s_0} \dreg    &
\shift R_k^{s_0}.         &
}
$$
Assume that (a) and (b) hold for some $c$,
and let us first prove that (a) holds for $c+1$.
Note that, by construction, $R_i^{(c+1)\ast}$ belongs
to $R_k^{\perp}=\shift^{-1}(R_k^{(1)})^{\perp}$ and so
does $R_j$, by (ii), so (a) makes sense.
Write $X$ for $X_{c+1}$.
Since $X$ belongs to $\add R_k$,
the space $\cat(X,R_j)$ vanishes and the
morphism $R_i^{(c+1)} \fl R_k^{s_{c+1}}\oplus R_j^{t_{c+1}}$
induces a morphism of triangles:
$$
\xymatrix{
X \dr \basp              &
R_i^{(c+1)} \dr   \bas   &
R_i^{(c+1)\ast}
\dr \basp^m              &
\shift X \basp           \\
R_k^{s_{c+1}} \dr        &
R_k^{s_{c+1}} \oplus
R_j^{t_{c+1}} \dr        &
R_j^{t_{c+1}} \dr^0      &
\shift R_k^{s_{c+1}}.
}
$$
We claim that $\underline{m}$
is a minimal left $\add R_j$-approximation
in $\cs$.
Let $f$ belong to $\cat(R_i^{(c+1)\ast},R_j)$.
The following diagram illustrates the proof:
$$
\xymatrix{
\susm R_i^{(c+2)} \bas_v          &
                                  &
                                  \\
R_i^{(c+1)} \dr^q \bas_u          &
R_i^{(c+1)\ast} \dr^p
\bas_m \bbdr^f                    &
\shift X                          \\
R_k^{s_{c+1}} \oplus
R_j^{t_{c+1}} \dr^(0.63)\pi_2 \bas
\ar@/_1.5pc/@{-->}[drr]_a         &
R_j^{t_{c+1}}
\ar@/_1pc/@{-->}[dr]^b          &
                                  \\
R_i^{(c+2)}                       &
                                  &
R_j
}
$$
where $\pi_2$ denotes the second projection.
Since the space $\cat(R_i^{(c+2)},\shift R_j)$
vanishes, we have $fqv = 0$ and there exists
a morphism $a$ such that $fq = au$. By (ii),
the morphism $a$ factors through $\pi_2$. Let $b$
be such that $a = b\pi_2$. We then have
$fq = b\pi_2 u = bmq$, and the morphism
$f - bm$ factors through $p$.
Since the object $X$ belongs to $\add R_k$, this
implies that $f-bm$ lies in the ideal $(\shift R_k)$.
That is $\underline{m}$ is a left $\add R_j$-approximation
in $\cs$.
Let $g\in\End_\cat(R_j^{t_{c+1}})$ be such that
$\underline{g}\,\underline{m} = \underline{m}$,
that is $gm-m$ belongs to the ideal $(\shift R_k)$.
This implies that the composition $(m-gm)q$
vanishes since $\cat(R_i^{(c+1)},\shift R_k) = 0$.
Let $h\in\cat(\shift X,R_j^{t_{c+1}})$ be such that
$gm = m+hp$. We have:
$gmq = mq + hpq =mq$. Since $mq = \pi_2 u$
is left minimal, the morphism $g$ is an isomorphism
in $\cat$, thus so is $\underline{g}$ in $\cs$.
Hence (a) holds for $c+1$.

Let us now prove that (b) holds for $c+1$.
By Lemma~\ref{lemma: CY approx}, Lemma~\ref{lemma: to k easy case}
and (a) for $c+1$, we have a minimal left
$\add R_j\oplus \shift R_k$
approximation of $R_i^{(c+1)\ast}$
in $\cat$ of the form
$\left[^m_{\,r} \right]$ for some $r:R_i^{(c+1)*}\gfl \shift R_k^{s_c}$,
which we complete to an exchange triangle
$$
R_i^{(c+1)\ast} \stackrel{\left[^m_{\,r} \right]}{\gfl}
R_j^{t_{c+1}} \oplus \shift R_k^{s_c} \gfl
R_i^{(c+2)\ast} \gfl \shift R_i^{(c+1)\ast}.
$$
Complete the commutative square
$$
\xymatrix{
R_i^{(c+1)} \dr^{u\hspace{.5cm}} \bas_q   &
R_j^{t_{c+1}} \oplus R_k^{s_{c+1}}
\bas^{\left[^{1 \; 0}_{0 \; 0} \right]}   \\
R_i^{(c+1)\ast}
\dr^{\left[^m_{\,r} \right]\hspace{.5cm}}     &
R_j^{t_{c+1}} \oplus \shift R_k^{s_c}
}
$$
to a commutative diagram
$$
\xymatrix{
R_i^{(c+1)} \dr^{u\hspace{.5cm}} \bas_q    &
R_j^{t_{c+1}} \oplus R_k^{s_{c+1}} \dr
\bas^{\left[^{1 \; 0}_{0 \; 0} \right]}    &
R_i^{(c+2)} \dr \bas                       &
\shift R_i^{(c+1)}                         \\
R_i^{(c+1)\ast} \bas
\dr^{\left[^m_{\,r} \right]\hspace{.5cm}}  &
R_j^{t_{c+1}} \oplus \shift R_k^{s_c}
\bas \dr                                   &
R_i^{(c+2)\ast} \dr \bas                   &
\shift R_i^{(c+1)\ast}                     \\
\shift X \dr \bas                          &
\shift R_k^{s_{c+1}} \oplus
\shift R_k^{s_c} \dr \bas                  &
\shift Y \dr^\eta \bas                     &
\shift^2 X                                 \\
\shift R_i^{(c+1)}                         &
\shift R_j^{t_{c+1}} \oplus
\shift R_k^{s_{c+1}}                       &
\shift R_i^{(c+2)}                         &
}
$$
whose rows and columns are triangles.
By construction, $R_i^{(c+2)\ast}$
belongs to $^\perp(\shift^2 R_k)$. Moreover,
$\cat(\shift R_i^{(c+2)},\shift^2 R_k)
\simeq \cat(R_i^{(c+2)},\shift R_k) = 0$.
Thus $\shift Y$ also belongs to the extension-closed
subcategory $^\perp(\shift^2 R_k)$, and the morphism
$\eta$ vanishes, since $X\in \add R_k$.
As a consequence, the triangle in the third row splits and $Y$
belongs to $\add R_k$. Define $X_{c+2}$ to be $Y$.
Then we see that (b) has been shown.

The statements (a') and (b') can be deduced from
(a) and (b) by duality, as we now explain.
Consider, in the category $\cat^\text{op}$,
the object $R_i \oplus R_j \oplus \shift R_k$.

It is a rigid object:
\begin{eqnarray*}
 \cat^\text{op}(R_i,\shift^\text{op} \shift R_k)
 & = &
 \cat^\text{op}(R_i,R_k) \\
 & = &
 \cat(R_k,R_i) \\
  & = & 0,
\end{eqnarray*}
and similarly, $\cat^\text{op}(R_j,\shift^\text{op} \shift R_k)=0$.
Moreover, it satisfies the assumptions we made to prove (a) and (b):
\begin{eqnarray*}
 \cat^\text{op}(\shift R_k,R_i)
 & = &
 \cat(R_i,\shift R_k) \\
 & = & 0.
\end{eqnarray*}

Note that
$\mu^\text{op}\shift R_k = R_k$,
$R_i^{(c)\text{op}} = R_i^{(-c)\ast}$ and
$R_i^{(c)\ast\text{op}} = R_i^{(-c)}$.
Therefore, there are triangles in $\cat$
$$
Y_{c+1} \longleftarrow R_i^{(-c-1)\ast}
\longleftarrow R_i^{(-c-1)} \longleftarrow
\susm Y_{c+1}
$$
with $Y_{c+1}$ in $\add \shift R_k$.
Let $X_{-c-1} = \susm Y_{c+1}$, to get
the triangles (b').

By (a) applied to $\cat^\text{op}$,
there are minimal right $\add R_j$ approximations
$R_i^{(-c)} \leftarrow R_j^{t_c^\text{op}}$
in $\cat_{R_k} = {}^\perp (\shift R_k) / (R_k)$.
This proves that we have $t_c^\text{op} = t_{-c-1}$.
Written in $\cat$, the exchange triangles
in $\cat^\text{op}$ for $R_i$ with respect to
$R_j\oplus \shift R_k$ are of the form:
$$
R_i^{(-c)\ast} \longleftarrow
R_j^{t_c^\text{op}}\oplus (\shift R_k)^{s_c^\text{op}}
\longleftarrow R_i^{(-c-1)\ast}
\longleftarrow \susm R_i^{(-c)\ast}.
$$
By Lemma~\ref{lemma: CY approx}, we thus have
minimal right $\add R_j$ approximations
$R_i^{(-c)\ast} \leftarrow R_j^{t_{-c-1}}$ in $\underline{\cat}$.
\end{proof}

\bibliographystyle{alpha}
\newcommand{\etalchar}[1]{$^{#1}$}

\end{document}